\documentclass{amsart}
\usepackage[utf8]{inputenc}
\usepackage{amsmath}
\usepackage[colorlinks]{hyperref}
\usepackage{amstext}
\usepackage{amscd}
\usepackage{amsmath}
\usepackage{latexsym}
\usepackage{url}
\usepackage{amsthm}
\usepackage{amssymb}
\usepackage[all]{xy}
\usepackage{mathrsfs}
\usepackage{color}
\usepackage{tikz}
\usepackage{indentfirst}

\newtheorem{theorem}{Theorem}[section]
\newtheorem{thmx}{Theorem}
\newtheorem{lemma}[theorem]{Lemma}
\newtheorem{corollary}[theorem]{Corollary}
\newtheorem{proposition}[theorem]{Proposition}

\theoremstyle{definition}
\newtheorem{definition}[theorem]{Definition}
\newtheorem{theorem-definition}[theorem]{Theorem-Definition}
\newtheorem{example}[theorem]{Example}

\theoremstyle{remark}
\newtheorem{remark}[theorem]{Remark}

\numberwithin{equation}{section}

\hypersetup{
     colorlinks,breaklinks,
     citecolor    = [RGB]{0,0,0},
     linkcolor =[RGB]{0,0,0}
}

\begin{document}
\title[{Generic Vanishing and Surface Classification}]{Generic Vanishing and Classification of Irregular Surfaces in Positive Characteristics }
\author{Yuan Wang}
\address{Department of Mathematics, University of Utah, 155 South 1400 East, Salt Lake City, UT 84112-0090, USA}
\email{ywang@math.utah.edu}
\thanks{The author was supported in part by the FRG grant DMS-\#1265261.}
\date{}
\dedicatory{}
\begin{abstract}
We establish a generic vanishing theorem for surfaces in characteristic $p$ that lift to $W_2(k)$ and use it for classification of surfaces of general type with Euler characteristic $1$ and large Albanese dimension. 
\end{abstract}
\maketitle
\section{Introduction}
The Enriques-Kodaira classification of surfaces was established by Enriques, Kodaira, Mumford and Bombieri in both zero and positive characteristics (cf. \cite{Enriques14}, \cite{Enriques49}, \cite{Kodaira64}, \cite{Kodaira66}, \cite{Kodaira68}, \cite{Mumford69}, \cite{BM76} and \cite{BM77}). A detailed classification of surfaces of general type, however, seems to be very difficult. Up to now the following progress has been made in characteristic $0$. First by Castelnuovo's inequality (cf. \cite[Theorem X.4]{Beauville96}) the Euler characteristic for any surface of general type must be strictly greater than $0$. When the Euler characteristic is $1$ Debarre gave an upper bound $p_g=q\le 4$ for the geometric genus and irregularity (cf. \cite{Debarre82} ). Beauville then discovered that the limit case $p_g=q=4$ corresponds to the product of two genus $2$ curves (cf. \cite{Beauville82}). Later based on works of Catanese, Ciliberto and Mendes Lopes (cf. \cite{CCM98}), Hacon and Pardini (cf. \cite{HP02}) and Pirola (cf. \cite{Pirola}) independently gave a complete classification for $p_g=q=3$ and Zucconi \cite{Zucconi03} classified the cases of $p_g=q=2$ with irrational pencil. In recent years it has become increasingly clear that generic vanishing is a fundamental tool in the study of irregular varieties (cf. \cite{HP02}).

In this paper inspired by \cite{HP02} and \cite{Beauville82} we establish a new type of generic vanishing for smooth surfaces that lift to the second Witt vectors $W_2(k)$ and we use it to prove two results on surface classification in positive characteristics.

Let us recall the generic vanishing theorem in its original form which was established by Green and Lazarsfeld in \cite{GL87}. Let $X$ be a smooth complex projective variety of dimension $n$. For an integer $i\ge 0$ let $V^i(\omega_X)$ be the subvariety of ${\rm Pic}^0(X)$ defined by
$$V^i(\omega_X)=\{P\in {\rm Pic}^0(X)|H^i(X,\omega_X\otimes P)\ne 0\}$$ 
and let
$$a:X\to A$$
be the Albanese morphism of $X$.
\begin{theorem}\cite[Theorem 1]{GL87}
$${\rm codim}(V^i(\omega_X),{\rm Pic}^0(X))\ge {\rm dim}(a(X))-n+i.$$
In particular if $P\in {\rm Pic}^0(X)$ is a general line bundle, then $H^i(X,\omega_X\otimes P)=0$ for $i>n-{\rm dim}(a(X)).$
\end{theorem}
Hacon \cite{Hacon04} and Pareschi and Popa \cite{PP11} have shown that if $X$ has maximal Albanese dimension then generic vanishing ($\omega_X$ being $GV_0$ to be precise) is equivalent to vanishing of $H^i(X,\omega_X\otimes a^*\widehat{L}^{\vee})$ for any $i>0$ and any sufficiently ample line bundle $L$ on $\hat{A}$. Note that $\widehat{L}^{\vee}$ is an ample vector bundle on $A$ with $h^0(\widehat{L}^{\vee})=1$ and ${\rm rk}(\widehat{L}^{\vee})=h^0(L)$. See Definition \ref{p17} for the definition.

It is known that generic vanishing theorems is in general not true for positive characteristic. Hacon and Kov\'{a}cs \cite{HK12} used a counter-example to the Grauert-Riemenschneider vanishing theorem to construct a counter-example to the generic vanishing theorem. But since the Grauert-Riemenschneider vanishing theorem is true for smooth surfaces in any characteristic one might expects that generic vanishing will hold in this context. We then prove

\begin{thmx}[Theorem \ref{1}]\label{i1}
Let $X$ be a smooth projective surface over an algebraically closed field $k$ of positive characteristic, $A$ an abelian variety and $a:X\to A$ a generically finite morphism. If $X$ lifts to $W_2(k)$ then $H^i(X,\Omega ^j_X\otimes P \otimes a ^*\widehat{L}^{\vee} )=0$ for any $i+j\geq 3$, $P\in {\rm Pic}^0(X)$ and any ample line bundle $L$ on $\hat{A}$. In particular for any $k>0$ and $P\in {\rm Pic}^0(X)$, $H^k(A,a_*(\omega_X\otimes P)\otimes Q)=0$ for general $Q\in {\rm Pic}^0(A)$.
\end{thmx}
In fact it turns out that roughly speaking, generic vanishing results are equivalent to results analogous to Kodaira vanishing. More precisely, we show that
\begin{thmx}[Corollary \ref{5}]
Let $a:X\to A$ be a generically finite morphism from a smooth projective surface $X$ to an abelian variety $A$, $\hat{A}$ the dual abelian variety of $A$ and $L$ an ample line bundle on $\hat{A}$. For an $m\in\mathbb{Z}^+$ let $\phi_{L^{\otimes m}}: \hat{A}\to A$ be the isogeny induced by $L^{\otimes m}$. Let $\hat{X}_m=X\times_A\hat{A}$ be the fiber product with respect to the morphisms $a$ and $\phi_{L^{\otimes m}}$, and let $\hat{a}_m:\hat{X}_m\to\hat{A}$ and $\varphi_m:\hat{X}_m\to X$ be the induced morphisms. Let $\widehat{L^{\otimes m}}$ be the ample vector bundle on $A$ defined as in Definition \ref{p17}.\\
{\rm (a)} If $H^i(\hat{X}_{m},\omega_{\hat{X}_m}\otimes \hat{a}_m^*(L^{\otimes m}\otimes P))=0$, $\forall P\in {\rm Pic}^0(\hat{A})$ and $i>0$, then $H^j(A,a_*\omega_X\otimes\widehat{L^{\otimes m}}^{\vee})=0$, $\forall j>0$. \\
{\rm (b)} If $H^i(A,a_*\omega_X\otimes\widehat{L^{\otimes m}}^{\vee})=0$, $\forall i>0$ and $m\gg 0$ then for any ample line bundle $M$,  $H^j(\hat{X}_{n},\omega_{\hat{X}_n}\otimes{\hat{a}_n}^*\phi_{nL}^*(M))=0$, $\forall j>0$ and $n\gg 0$.
\end{thmx}
The first application of Theorem \ref{i1} is about the existence of irrational pencils on surfaces with Euler characteristic $0$.
\begin{thmx}[Proposition \ref{f1} and Theorem \ref{f2}]\label{i2}
Let $X$ be a smooth minimal projective surface of maximal Albanese dimension which lifts to $W_2(k)$, then $\chi(\omega_X)\ge 0$. If moreover $\chi(\omega_X)=0$ and the Picard variety of $X$ has no supersingular factors then either $X$ has an irrational pencil of genus  $\ge {\rm dim}(V^1(\omega_X))\ge 1$ or $X$ is an abelian surface. 
\end{thmx} 
The second application, which is our main result about classification of surfaces of general type, is as follows. 
\begin{thmx}[Theorem \ref{c4}]
Let $X$ be a smooth minimal projective surface of general type over an algebraically closed field $k$ of characteristic $\ge 11$ and $\chi(\mathcal{O}_X)=1$. Denote the Albanese morphism as $a: X\to A$. Assume that $X$ is of maximal Albanese dimension, lifts to $W_2(k)$, its Picard variety has no supersingular factors and $a$ is separable. If ${\rm dim}(A)=4$ then $X=C_1\times C_2$ where $C_1$ and $C_2$ are smooth curves and $g(C_1)=g(C_2)=2$.
\end{thmx}
We will explain the meaning of the condition ${\rm dim}(A)=4$ later (Remark \ref{c12}). Note that even though we established a good generic vanishing theorem the proof of Theorem \ref{i2}, inspired by \cite{Beauville82} and \cite{HP02}, requires many new ideas. A considerable difficulty is that there is no obvious irrational pencils on $X$ and a large part of Hodge theory as well as other characteristic $0$ techniques used in \cite{Beauville82} and \cite{HP02} to construct irrational pencils fail in positive characteristics. Therefore a detailed analysis of $V^1(\omega_X)$ and the Fourier-Mukai transforms of various sheaves is made through Propositions \ref{c1}-\ref{c2}. 
\subsection*{Acknowledgements} 
The author would like to thank his advisor Professor Christopher Hacon for suggesting this project and a lot of inspiring discussions, support and encouragement. He would also like to thank Tong Zhang for suggesting and proving the inequality \eqref{c11}. Finally he would like to thank the referee for many valuable suggestions.
\section{Conventions, notations and preliminaries}
We fix an algebraically closed field $k$ and assume that all the schemes we will discuss are over $k$. We make no restriction on the characteristic of $k$ unless otherwise stated. 
\subsection{Derived categories}
For any scheme $X$ of dimension $n$ we denote by $D(X)$ the derived category of $\mathcal{O}_X$-modules and denote by $D_{c}(X)$ (resp. $D_{qc}(X)$) the full subcategory of $D(X)$ consisting of complexes whose cohomologies are coherent (resp. quasi-coherent). We also denote the dualizing complex by $\omega_X^{\cdot}$ and define the dualizing functor $D_X$ by $D_X(F)=R\mathcal{H}om(F,\omega_X^{\cdot}[n])$, $\forall F\in D_{qc}(X)$. We will use projection formula and Grothendieck duality in the following forms.
\begin{theorem}[Projection formula]
Let $f:X\to Y$ be a morphism of quasi-compact separated schemes. Let $F\in D_{qc}(X)$ be a sheaf and $G\in D_{qc}(Y)$ be a locally free sheaf. Then there is an isomorphism
$$Rf_*(F)\otimes_{\mathcal{O}_Y}G\xrightarrow{\cong}Rf_*(F\otimes_{\mathcal{O}_X}f^*G).$$ 
\end{theorem}
\begin{theorem}[Grothendieck duality]
Let $f:X\to Y$ be a proper morphism of quasi-projective varieties, then
$$Rf_*D_X(F)=D_Y(Rf_*(F)),\ \forall F\in D_{qc}(X).$$
\end{theorem}

We will also need Grauert-Riemenschneider vanishing theorem which is known to hold for smooth surfaces in any characteristic. We generalize the original statement in the following way:
\begin{theorem}[Grauert-Riemenschneider vanishing theorem]\label{p1}
Let $f:X\to Y$ be a projective and generically finite morphism, $Y$ normal and quasi-projective and $X$ a smooth surface. Then $R^1f_*(\omega_X\otimes P)=0$ for any $P\in {\rm Pic}^0(X)$.
\end{theorem}  
\begin{proof}
The original statement and proof can be found in \cite{Kollar07} as Theorem 2.20.1 which says that $R^1f_*\omega_X=0$. The proof also works for $\omega_X\otimes P$.
\end{proof}
\subsection{Abelian varieties, Fourier-Mukai transform and generic vanishing theorems}
\begin{definition}
Let $A$ be an abelian variety. For a subvariety $X\subseteq A$ we say that $X$ \emph{generates} $A$ if $X$ is not contained in any proper abelian subvariety of $A$. 
\end{definition}
\begin{definition}
Let $X$ be a projective variety and let $a: X\to A$ be the Albanese morphism of $X$. We say that $X$ is \emph{of maximal Albanese dimension (mAd)} if ${\rm dim}(X)={\rm dim}(a(X))$. 
\end{definition}
Let $A$ be an abelian variety. Let $\hat{A}$ be its dual abelian variety and $p_A: A\times \hat{A}\to A$ and $p_{\hat{A}}:A\times \hat{A}\to \hat{A}$ be the projection morphisms. Let $\mathcal{P}$ be the Poincar\'{e} line bundle on $A\times \hat{A}$. We define the Fourier-Mukai transform $R\hat{S}:D(A)\to D(\hat{A})$ and $RS:D(\hat{A})\to D(A)$ with respect to the kernel $\mathcal{P}$ by
$$R\hat{S}(\cdot)=Rp_{\hat{A},*}(p_A^*(\cdot)\otimes\mathcal{P}),\ \ \ RS(\cdot)=Rp_{A,*}(p_{\hat{A}}^*(\cdot)\otimes\mathcal{P}).$$
Next we recall some facts proven in \cite{Mukai81}.
\begin{theorem}\label{p10}\cite[Theorem 2.2]{Mukai81}
The following isomorphisms of functors hold on $D_{qc}(A)$ and $D_{qc}(\hat{A})$:
$$RS\circ R\hat{S}=(-1_A)^*[-g]$$
$$R\hat{S}\circ RS=(-1_{\hat{A}})^*[-g]$$
where $[-g]$ means shifting by $g$ steps to the right and $-1_A$ means the inverse morphism on $A$. 
\end{theorem}
\begin{lemma}\cite[(3.1)]{Mukai81}\label{p4}
For any $x\in A$ and $y\in \hat{A}$ the following isomorphisms hold on $D_{qc}(A)$ and $D_{qc}(B)$ respectively:
$$RS\circ T_y^*\cong (\otimes P_{-y})\circ RS$$
$$RS\circ(\otimes P_x)\cong T_x^*\circ RS,$$
where $P_x=\mathcal{P}|_{\{x\}\times\hat{A}}$, $P_y=\mathcal{P}|_{A\times \{y\}}$ and $T_x$ and $T_y$ are translations by $x$ and $y$ on $A$ and $\hat{A}$ respectively.
\end{lemma}
\begin{lemma}\label{p11}\cite[(3.4)]{Mukai81}
Let $A$ and $B$ be abelian varieties, $\varphi:A\to B$ an isogeny and $\hat{\varphi}:\hat{B}\to\hat{A}$ the dual isogeny of $\varphi$. Then the following isomorphisms hold on $D_{qc}(B)$ and $D_{qc}(A)$ respectively:
$$\varphi^*\circ RS_B=RS_A\circ\hat{\varphi}_*$$
$$\varphi_*\circ RS_A=RS_B\circ\hat{\varphi}^*.$$
\end{lemma}
\begin{proposition}\label{p13}\cite[Proposition 3.11 (1)]{Mukai81}
Let $A$ be an abelian variety, $L$ an ample line bundle on $A$ and $\phi_L: A\to \hat{A}$ the isogeny induced by $L$. Then 
$$\phi_L^*(\widehat{L})=\bigoplus^{h^0(A,L)}L^{\vee}.$$ 
\end{proposition}
Recall the following 
\begin{definition}
For an abelian variety $A$ and a line bundle $L$ on $A$ we define
$$K(L)=\{x\in A|T_x^*(L)\cong L\}$$
where $T_x$ is the translation morphism with respect to $x$. We say that $L$ is \emph{non-degenerate} if $K(L)$ is finite, otherwise we say that $L$ is \emph{degenerate}.
\end{definition}
\begin{theorem-definition}\label{p18}
For any non-degenerate line bundle $L$ on $A$ by the vanishing theorem in \cite[Section 16]{Mumford12}, there exists a unique $i\in \mathbb{Z}$, $0\le i\le {\rm dim}(A)$ such that $H^i(X,L)\ne 0$ and we denote this $i$ as $i(L)$.
\end{theorem-definition}
\begin{proposition}\cite[Proposition (9.18)]{MvdG}
$i(L)=0$ for any ample line bundle $L$ on $A$.
\end{proposition}
The next proposition is very useful in later sections.
\begin{proposition}\label{p7}
If $L$ is a line bundle on an abelian variety $A$, then there is a unique integer $i\in \mathbb{Z}$, $0\le i\le {\rm dim}(A)$ such that $R\hat{S}(L)=R^i\hat{S}(L)[-i]$ is a sheaf and its support is an abelian subvariety of $\hat{A}$. If $L$ is non-degenerate then this integer is equal to the integer $i(L)$ in Theorem-Definition \ref{p18}.
\end{proposition}
\begin{proof}
If $L$ is nondegenerate, then by \cite[p.145 Theorem]{Mumford12}, $i(L)$ can be computed as the number of positive roots of $P(n)=\chi(M^n\otimes L)$ for arbitrary ample line bundles $M$. Since for any $P_y$ we have $\chi(M^n\otimes L)=\chi(M^n\otimes L\otimes P_y)$ we know that $i(L\otimes P_y)$ is independent of $P_y$. By cohomology and base change 
\begin{align*}
R^i\hat{S}(L)\otimes k(y)=& R^ip_{2,*}(p_1^*L\otimes \mathcal{P})\otimes k(y)\cong H^i(A\times\{y\},(p_1^*L\otimes \mathcal{P})|_{A\times\{y\}}) \\
=& H^i(A,L\otimes P_y).
\end{align*}
So $R^i\hat{S}(L)\ne 0$ if and only if $i=i(L)$. Therefore we have
$$R\hat{S}(L)=R^{i(L)}\hat{S}(L)[-i(L)]$$ 
and it is supported on $\hat{A}$.

If $L$ is degenerate let $K(L)^0$ be the connected component of $K(L)$ containing the origin. Let $Z=(K(L)^0)_{\rm red}$ which by \cite[Proposition 5.31]{MvdG} is an abelian subvatiety of $A$. By Poincar\'{e}'s complete reducibility theorem  (cf. \cite[p.160 Theorem]{Mumford12}) there is an isogeny $\varphi:Y\times Z\to A$ where $Y$ is an abelian subvariety of $A$. By the proof of \cite[Proposition 9.27]{MvdG} there is a non-degenerate line bundle $L_Z$ on $Z$ such that $\varphi^*L=p_Z^*L_Z\otimes P_x$ for some $x\in \hat{Y}\times\hat{Z}$. By Lemma \ref{p4}, Lemma \ref{p11} and \cite[Exercise 5.13]{Huybrechts06} 
\begin{align*}
\hat{\varphi}_*R\hat{S}_A(L)=& R\hat{S}_{Y\times Z}(\varphi^*L)=R\hat{S}_{Y\times Z}(p_Z^*L_Z\otimes P_x)=T_x^*R\hat{S}_{Y\times Z}(p_Z^*L_Z) \\
=& T_x^*R\hat{S}_{Y\times Z}(L_Z\boxtimes \mathcal{O}_Y)=T_x^*(R\hat{S}_Z(L_Z)\boxtimes R\hat{S}_Y(\mathcal{O}_Y)) \\
=& T_x^*(R^{i(L_Z)}\hat{S}_Z(L_Z)[-i(L_Z)]\boxtimes R^{{\rm dim}(Y)}\hat{S}_Y(\mathcal{O}_Y)[-{\rm dim}(Y)]). 
\end{align*} 
By the fact that $\hat{\varphi}$ is finite we get
$$R\hat{S}_A(L)=R^{i(L_Z)+{\rm dim}(Y)}\hat{S}_A(L)[-(i(L_Z)+{\rm dim}(Y))].$$
Observe that 
$$0\le {\rm dim}(Y)\le i(L_Z)+{\rm dim}(Y)\le {\rm dim}(Z)+{\rm dim}(Y)\le {\rm dim}(A).$$
Moreover since $R^{i(L_Z)}\hat{S}_Z(L_Z)$ is supported on $\hat{Z}$ and $R^{{\rm dim}(Y)}\hat{S}_Y(\mathcal{O}_Y)$ is supported on the origin $e_{\hat{Y}}$ on $\hat{Y}$ we know that $R\hat{S}_A(L)$ is supported on $\hat{\varphi}^{-1}\circ T_x(\hat{Z}\times {e_{\hat{Y}}})$ which is an abelian subvariety in $\hat{A}$. 
\end{proof}
\begin{definition}\label{p17}
For a nondegenerate line bundle $L$ on an abelian variety $A$ we have seen that the integer constructed in Proposition \ref{p7} is compatible with the $i(L)$ in Theorem-Definition \ref{p18}. So we define this integer to be the \textit{index} of $L$ and still use $i(L)$ to denote it. Moreover we denote $R^{i(L)}\hat{S}(L)$, viewed as a sheaf on $\hat{A}$, by $\widehat{L}$. 
\end{definition}
\begin{remark}
Note that if $L$ is an ample line bundle then $\widehat{L}$ is locally free (see \cite[Example 2.2]{PP03} and \cite[right before Corollary 2.4]{Mukai81}).
\end{remark}
\begin{remark}
In the situation of Definition \ref{p17}, by \cite[(3.8)]{Mukai81} we have 
\begin{align}\label{p16}
\widehat{L}^{\vee}=(-1_A)^*\widehat{L^{\vee}}.
\end{align}
\end{remark}
Next we will mention some concepts and facts related to generic vanishing theorems.
\begin{definition}\label{p3}
Let $A$ be an abelian variety of dimension $n$ and $\mathcal{F}$ a coherent sheaf on $A$. We define 
$V^i(\mathcal{F})$, the \emph{cohomology support loci}, as $$V^i(\mathcal{F})=\{P\in {\rm Pic}^0(X)|h^i(A,\mathcal{F}\otimes P)> 0\}.$$ $\mathcal{F}$ is said to \emph{satisfy Generic Vanishing with index $-k$}, or to be $GV_{-k}$, if codim$_{{\rm Pic}^0(X)}V^i(\mathcal{F})\ge i-k$ for all $i\ge 0$. 
\end{definition}

The following generic vanishing theorem can be deduced from \cite[Theorem 1.2]{Hacon04} and \cite[Theorem A]{PP11}.
\begin{theorem}\label{p5}
With the notation as in \ref{p3}, the following are equivalent: \\
(1) $\mathcal{F}$ is $GV_{0}$. \\
(2) $H^i(A,\mathcal{F}\otimes\widehat{L}^{\vee})=0$ for any $i>0$ and any sufficiently ample line bundle $L$ on $X$.\\
(3) $R\hat{S}(D_A(\mathcal{F}))=R^0\hat{S}(D_A(\mathcal{F}))$.
\end{theorem}
\begin{remark}
In this paper we use the notion of ``sufficiently ample" line bundle on a variety $X$ to mean, given any ample line bundle $L$, a power $L^{\otimes m}$ with $m\gg 0$.
\end{remark}
The following fact due to Pink and Roessler (of \cite{PR04}) will also play an important role.
\begin{proposition}\label{p19}
Let $X$ be a projective variety of dimension $n$ over $k$ and assume that the Picard variety of $X$ has no supersingular factors. Let
$$S_m^{i,j}(X)=\{P\in {\rm Pic}^0(X)|h^{i,j}(X,P)\ge m\}$$
for any $i,j,m\ge 0$. Then $S_m^{i,j}$ is completely linear, i.e. its irreducible components are translates of abelian subvarieties by torsion elements in ${\rm Pic}^0(X)$. In particular, $V^1(\omega_X)=S_1^{1,n}(X)$ is completely linear.
\end{proposition}
\begin{proof}
By \cite[Proposition 4.1 and Proposition 3.1]{PR04}.
\end{proof}

\subsection{Lifting properties of algebraic varieties}
For an algebraically closed field $k$ of characteristic $p>0$ let $W_2(k)$ be the ring of the second Witt vectors of $k$ (for details see \cite[Example 8.8]{EV92}). Let $X$ be a scheme over $k$ and denote $S={\rm Spec}(k)$ and $\widetilde{S}={\rm Spec}(W_2(k))$. A \textit{lifting} of $X$ to $\widetilde{S}$ is a scheme $\widetilde{X}$, flat over $\tilde{S}$, such that $X=\widetilde{X}\times_{\widetilde{S}}S$. We say that $X$ \textit{lifts to} $W_2(k)$ or $X$ \textit{is liftable to} $W_2(k)$ if $X$ has a lifting to $W_2(k)$. 

We define $X'$ by $X'=X\times_{S}S$ via the Frobenius morphism $F:S\to S$ and denote by $F'$ the morphism $X\to X'$ naturally defined by $X\to S$ and the Frobenius morphism $F:X\to X$.
\begin{theorem}\cite[Th\'eor\`eme 2.1]{DI87}\label{p12}
With the notation above, if $X$ lifts to $W_2(k)$ then the following isomorphism holds in $D(X')$:
$$\varphi:\bigoplus_{i<p}\Omega_{X'/S}^i\xrightarrow{\cong}\tau_{<p}F'_*\Omega_{X/S}^{\cdot}.$$
Moreover the $i$-th cohomology of $\varphi$ is the Cartier isomorphism for $i<p$.
\end{theorem}

\subsection{Preliminaries on surfaces in positive characteristics}
\begin{definition}\label{Irreg&Euler}
Let $X$ be a surface. We define the \emph{irregularity} $q(X)$ of $X$ to be $h^1(X,\mathcal{O}_X)$, and define the \emph{Euler characteristic} of $X$ to be $\chi(\mathcal{O}_X)$. It is easy to see that for a surface $X$ we have $\chi(\mathcal{O}_X)=\chi(\omega_X)$.
\end{definition}
\begin{definition}
An \emph{irrational pencil of genus $g$} on a surface $X$ is a fibration $p: X\to B$ where $B$ is a smooth curve of genus $g\ge 1$.
\end{definition}
\begin{proposition}[Bogomolov–Miyaoka–Yau inequality in positive characteristics]\label{p6}\cite[Theroem 13]{Langer15}
Let $X$ be a minimal minimal surface of general type. If ${\rm char}(k)\ge 3$ and $X$ is liftable to $W_2(k)$, then 
$$K_X^2\le 9\chi(\mathcal{O}_X).$$
\end{proposition}
The following proposition is well-known, but since its proof is very short we still would like to include it.
\begin{proposition}\label{p8}
If $S$ is an elliptic surface over $k$, then $S$ is not of general type. 
\end{proposition}
\begin{proof}
Let $f:S\to B$ be an elliptic fibration of $S$ and $E$ a general fiber. If $S$ is of general type by the adjunction formula 
$$K_E=(K_S+E)|_E=K_S|_E$$
is big. However since $E$ is an elliptic curve we know $K_E=\mathcal{O}_E$ is not big. Contradiction.
\end{proof}

\subsection{Some facts about torsion-free sheaves}
\begin{lemma}\label{p2}
Let $Z$ be an irreducible variety and suppose that $L$ is a line bundle and $\mathcal{F}$ is a torsion-free coherent sheaf on $Z$. If $\varphi:L\to \mathcal{F}$ is a nonzero morphism then $\varphi$ is injective.
\end{lemma}
\begin{proof}
We know that ${\rm Supp}({\rm im}(\varphi))$ is a closed subset. On the other hand $\mathcal{F}$ is torsion-free so it cannot contain a torsion sheaf, so ${\rm Supp}({\rm im}(\varphi))=Z$. Next we restrict the morphism to an affine open set $U={\rm Spec}(A)$ and suppose $L|_U=\tilde{M}=\tilde{(Ae)}$ and $\mathcal{F}|_U=\tilde{N}$ where $M$ and $N$ are $A$-modules and we denote by $\phi:M\to N$ the induced homomorphism. Suppose $\varphi(e)=n\in N$ then by the above argument $n\ne 0$. If there exists a nonzero $a$ in $A$ such that $\varphi(ae)=a\cdot n=0$ then it contradicts the fact that $N$ is torsion-free.
\end{proof}
\begin{lemma}\label{p15}
Let $X$ and $Y$ be projective varieties and $f: X\to Y$ a finite morphism. If $\mathcal{F}$ is a torsion-free sheaf on $Y$ then $f^*\mathcal{F}$ is also torsion-free.
\end{lemma}
\begin{proof}
We may assume that $X$ and $Y$ are affine and $f$ is induced by a ring homomorphism $\psi: R\to S$ where $X={\rm Spec}(S)$, $Y={\rm Spec}(R)$.

Now the claim becomes that if $M$ is a torsion-free $R$-module then $M\otimes_R S$ is a torsion-free $S$-module. Since $f$ is finite we can assume that $S$ is generated by $\{s_i\}$ as an $R$-module. If the claim is not true then there exists $0\ne \sum_i m_i\otimes s_i\in M\otimes_RS$ and $0\ne s'\in S$ such that
$$s'(\sum_i m_i\otimes s_i)=\sum_i m_i\otimes s_is'=0.$$
By \cite[Lemma 6.4]{Eisenbud99} there exist $m_j'\in M$ and $a_{ij}\in R$ such that 
$$\sum_{j}a_{ij}m_j'=m_i, \forall i$$
and
$$\sum_{i}a_{ij}s_is'=0, \forall j$$
Since $s'\ne 0$ and $S$ is integral by assumption it follows that $\sum_{i}a_{ij}s_i=0$. By \cite[Lemma 6.4]{Eisenbud99} again we get $\sum_i m_i\otimes s_i=0$ which is a contradiction.
\end{proof}

\section{Generic vanishing for surfaces} \label{s3}
\begin{theorem}\label{1}
Let $X$ be a smooth projective surface over an algebraically closed field $k$ of positive characteristic, $A$ an abelian variety and $a:X\to A$ a generically finite morphism. If $X$ lifts to $W_2(k)$ then $H^i(X,\Omega ^j_X\otimes P \otimes a^*\widehat{L}^{\vee} )=0$ for any $i+j\geq 3$, $P\in {\rm Pic}^0(X)$ and ample line bundle $L$ on $\hat{A}$. In particular for any $k>0$ and $P\in {\rm Pic}^0(X)$, $H^k(A,a_*(\omega_X\otimes P)\otimes Q)=0$ for general $Q\in {\rm Pic}^0(A)$.
\end{theorem}
\begin{remark}\label{codimV1}
In particular in Theorem \ref{1} if we let $P=\mathcal{O}_X$ we get $H^k(A, a_*\omega_X\otimes Q)=0$ for $k\ge 1$ and general $Q\in {\rm Pic}^0(A)$. By Theorem \ref{p5} this is equivalent to $H^k(X, \omega_X\otimes a^*Q)=0$, in particular $V^1(\omega_X)\ne{\rm Pic}^0(X)$. Then by applying the semicontinuity theorem (cf. \cite[Ch. III Theorem 12.8]{Hartshorne77}) with $f:X\times {\rm Pic}^0(X)\to {\rm Pic}^0(X)$ and $\mathscr{F}=K_{X\times {\rm Pic}^0(X)}\otimes\mathcal{P}$, where $\mathcal{P}$ is the Poincar\'{e} line bundle on $X\times {\rm Pic}^0(X)$, we get that $V^1(\omega_X)$ is a proper closed subset of ${\rm Pic}^0(X)$. This means that ${\rm codim}_{{\rm Pic}^0(X)}V^1(\omega_X)\ge 1$. Moreover by Serre duality $H^2(X,\omega_X\otimes P)=H^0(X,P^{\vee})^{\vee}$ and it is nonzero iff $P=\mathcal{O}_X$, so we also have ${\rm codim}_{{\rm Pic}^0(X)}V^2(\omega_X)\ge 2$. Therefore we see that Theorem \ref{1} actually implies that $\omega_X$ is $GV_0$.
\end{remark}

The proof of Theorem \ref{1} is in the spirit of the article of Deligne and Illusie \cite{DI87}. Let $\hat{A}$ be the dual Abelian variety of $A$, $L$ be an ample line bundle on $\hat{A}$, $\phi_{L}:\hat{A}\to A$ be the isogeny induced by $L$. We use $F$ to denote the $k$-linear Frobenius morphism. To prove the theorem we need the following lemmas.
\begin{lemma}\label{2}
Suppose $\mathcal{F}$ is a coherent sheaf on $A$. Then there exists an $e_0$ such that for any $e\ge e_0$, $H^i(A,\mathcal{F}\otimes F^{e,*}(\widehat{L}^{\vee})\otimes Q)=0$ for any $i>0$ and $Q\in {\rm Pic}^0(A)$.
\end{lemma}
\begin{proof}
By Fujita's vanishing theorem (cf. \cite[Theorem 1.4.35]{Lazarsfeld04}) and Proposition \ref{p13} there exists an $e_0$ such that for any $e\ge e_0$,
\begin{align*}
0&=H^i(\hat{A},\phi_{L}^*(\mathcal{F})\otimes(\bigoplus^{h^0(\hat{A},L)}L^{p^e})\otimes Q) \\
&=H^i(\hat{A},\phi_{L}^*(\mathcal{F})\otimes F^{e,*}\phi_{L}^*(\widehat{L}^{\vee})\otimes Q) \\
&=H^i(\hat{A},\phi_{L}^*(\mathcal{F})\otimes\phi_{L}^*F^{e,*}(\widehat{L}^{\vee})\otimes Q) \\
&=H^i(\hat{A},\phi_{L}^*(\mathcal{F}\otimes F^{e,*}(\widehat{L}^{\vee}))\otimes Q). 
\end{align*}
for any $i>0$ and $Q\in {\rm Pic}^0(\hat{A})$.  Then by cohomology and base change and Lemma \ref{p11} we know that 
\begin{align}
RS(\phi_{L}^*(\mathcal{F}\otimes F^{e,*}(\widehat{L}^{\vee})))=\widehat{\phi_{L}}_*R\hat{S}(\mathcal{F}\otimes F^{e,*}(\widehat{L}^{\vee})) \label{6}
\end{align}
is a sheaf in degree $0$. Moreover since $\widehat{\phi_{L}}$ is finite then $R\hat{S}(\mathcal{F}\otimes F^{e,*}(\widehat{L}^{\vee}))$ is also a sheaf in degree $0$. By cohomology and base change this implies that $H^i(A,\mathcal{F}\otimes F^{e,*}(\widehat{L}^{\vee})\otimes Q)=0$ for any $i>0$ and $Q\in {\rm Pic}^0(A)$.
\end{proof}
\begin{lemma}\label{7}
Let $P$ be a line bundle in ${\rm Pic}^0(X)$. We denote the support of $\phi_{L}^*(a_*(\Omega_X^1\otimes P))$ by $Z$. Then $\phi_{L}^*(a_*(\Omega_X^1\otimes P))$ is a torsion-free sheaf on $Z$.
\end{lemma}
\begin{proof}
$\Omega_X^1\otimes P$ is torsion-free by definition, and $a_*(\Omega_X^1\otimes P)$ is torsion-free on $a(X)$ because any push-forward of a torsion-free sheaf is torsion-free on the image. Then $\phi_{L}^*(a_*(\Omega_X^1\otimes P))$ is torsion-free on $Z$ by Lemma \ref{p15} and finiteness of $\phi_L$. 
\end{proof} 
\begin{lemma}\label{3}
There exists $e_0$ such that for any $e\ge e_0$, $i>0$ and $P\in {\rm Pic}^0(X)$, $H^i(X,\omega_X\otimes P\otimes F^{e,*}a^*\widehat{L}^{\vee})=0$ and $H^2(X,\Omega_X^1\otimes P\otimes F^{e,*}a^*\widehat{L}^{\vee})=0$.
\end{lemma}
\begin{proof}
First we prove that $H^i(X,\omega_X\otimes P\otimes F^{e,*}a^*\widehat{L}^{\vee})=0$. By Theorem \ref{p1} we have $R^ia_*(\omega_X\otimes P)=0$ for any $i>0$. Hence we have 
$$H^i(X,\omega_X\otimes P\otimes F^{e,*}a^*\widehat{L}^{\vee})=H^i(X,\omega_X\otimes P\otimes a^*F^{e,*}\widehat{L}^{\vee})=H^i(A,a_*(\omega_X\otimes P)\otimes F^{e,*}\widehat{L}^{\vee})$$ 
where the first equality is by the commutativity of $a$ and $F$ and the second equality is by the Projection Formula and degeneration of a Leray spectral sequence (cf. \cite[Exercise III.8.1]{Hartshorne77}). The claim then follows from  Lemma \ref{2} after we replace $\mathcal{F}$ and $Q$ by $a_*(\omega_X\otimes P)$ and $\mathcal{O}_A$ respectively.

Next we will prove that $H^2(X,\Omega_X^1\otimes P\otimes F^{e,*}a^*\widehat{L}^{\vee})=0$ for $e\gg 0$ which by Serre duality is equivalent to $H^0(X,\Omega_X^1\otimes P^{\vee}\otimes F^{e,*}a^*\widehat{L})=0$. We first prove that there exists an $e_0$ such that for any $e\ge e_0$ and any $Q\in {\rm Pic}^0(\hat{A})$, $H^0(\hat{A},\phi_{L}^*(a_*(\Omega_X^1\otimes P^{\vee}))\otimes(L^{-p^e})\otimes Q)=0$. This is because every nonzero element in $H^0(\hat{A},\phi_{L}^*(a_*(\Omega_X^1\otimes P^{\vee}))\otimes(L^{-p^e})\otimes Q)$ corresponds to a nonzero morphism $L^{p^e}\otimes Q^{\vee}\to \phi_{L}^*(a_*(\Omega_X^1\otimes P^{\vee}))$. We claim that this morphism is injective after we restrict it to an irreducible component $Z_0$ of $Z$. Indeed, $L^{p^e}\otimes Q^{\vee}|_{Z_0}$ is a line bundle which is torsion-free of rank $1$ and $\phi_{L}^*(a_*(\Omega_X^1\otimes P^{\vee}))|_{Z_0}$ is torsion-free by Lemma \ref{7}, so the injectivity follows from Lemma \ref{p2}. Since ${\rm dim}(Z_0)=2$, after taking $H^0(Z_0,\cdot)$ on both sides of the morphism $L^{p^e}\otimes Q^{\vee}|_{Z_0}\to \phi_{L}^*(a_*(\Omega_X^1\otimes P^{\vee}))|_{Z_0}$ we see that for any $Q\in {\rm Pic}^0(X)$, $h^0(Z_0,L^{p^e}\otimes Q^{\vee}|_{Z_0})$ goes to infinity as $e\gg 0$ while $h^0(Z_0,\phi_{L}^*(a_*(\Omega_X^1\otimes P^{\vee})|_{Z_0}))$ is constant and the induced homomorphism is still injective. This is a contradiction which completes the proof that $H^0(\hat{A},\phi_{L}^*(a_*(\Omega_X^1\otimes P^{\vee}))\otimes((-L)^{p^e})\otimes Q)=0$. 

Now by cohomology and base change $H^0(\hat{A},\phi_{L}^*(a_*(\Omega_X^1\otimes P^{\vee}))\otimes((-L)^{p^e})\otimes Q)=0$ implies that $R^0S(\phi_{L}^*(a_*(\Omega_X^1\otimes P^{\vee})\otimes F^{e,*}(\widehat{L})))=0$. By Lemma \ref{p11} we see that 
$$RS(\phi_{L}^*(a_*(\Omega_X^1\otimes P^{\vee})\otimes F^{e,*}(\widehat{L})))=(R\widehat{\phi_{L}}_*\circ R\hat{S})(a_*(\Omega_X^1\otimes P^{\vee})\otimes F^{e,*}(\widehat{L})),$$
and after taking cohomology in degree $0$ we have 
$$0=R^0(\widehat{\phi_{L}}_*\circ\hat{S})(a_*(\Omega_X^1\otimes P^{\vee})\otimes F^{e,*}(\widehat{L})).$$
Since $\widehat{\phi_{L}}$ is finite the following Grothendieck spectral sequence
\begin{align*}
E_2^{ij}&=R^i\widehat{\phi_{L}}_*R^j\hat{S}(a_*(\Omega_X^1\otimes P^{\vee})\otimes F^{e,*}(\widehat{L})) \\
&\Rightarrow R^{i+j}(\widehat{\phi_{L}}_*\circ\hat{S})(a_*(\Omega_X^1\otimes P^{\vee})\otimes F^{e,*}(\widehat{L}))
\end{align*}
degenerates at $E_2$, in particular  
\begin{align*}
0=& R^0(\widehat{\phi_{L}}_*\circ\hat{S})(a_*(\Omega_X^1\otimes P^{\vee})\otimes F^{e,*}(\widehat{L})) \\
=& R^0\widehat{\phi_{L}}_*R^0\hat{S}(a_*(\Omega_X^1\otimes P^{\vee})\otimes F^{e,*}(\widehat{L})) \\
=& \widehat{\phi_{L}}_*R^0\hat{S}(a_*(\Omega_X^1\otimes P^{\vee})\otimes F^{e,*}(\widehat{L})).
\end{align*}
Therefore $R^0\hat{S}(a_*(\Omega_X^1\otimes P^{\vee})\otimes F^{e,*}(\widehat{L}))=0$, and then by cohomology and base change $H^0(A, a_*(\Omega_X^1\otimes P^{\vee})\otimes F^{e,*}\widehat{L}\otimes Q)=0$ for all $Q\in U$ where $U$ is an open set of ${\rm Pic}^0(A)$. Now if we can show that $h^0(X,\Omega_X^1\otimes P^{\vee}\otimes a^*(F^{e,*}\widehat{L}\otimes P))$ is constant with respect to $P$ for $e\gg 0$, then we have 
\begin{align*}
0&=H^0(A,a_*(\Omega_X^1\otimes P^{\vee})\otimes F^{e,*}\widehat{L}\otimes Q_0) \\
&=H^0(X,\Omega_X^1\otimes P^{\vee}\otimes a^*(F^{e,*}\widehat{L}\otimes Q_0)) \\
&=H^0(X,\Omega_X^1\otimes P^{\vee}\otimes a^*(F^{e,*}\widehat{L}\otimes Q)),
\end{align*}
for any $Q_0\in U$ and any $Q\in {\rm Pic}^0(A)$, and after taking $P=\mathcal{O}_A$ we are done.

To prove that $h^0(X,\Omega_X^1\otimes P^{\vee}\otimes a^*(F^{e,*}\widehat{L}\otimes Q))$ is constant with respect to $Q$ it suffices to prove that $h^1(X,\Omega_X^1\otimes P^{\vee}\otimes a^*(F^{e,*}\widehat{L}\otimes Q))$ and $h^2(X,\Omega_X^1\otimes P^{\vee}\otimes a^*(F^{e,*}\widehat{L}\otimes Q))$ are both constant. By Lemma \ref{2} we know $H^i(A,a_*(\Omega_X^1\otimes P^{\vee})\otimes F^{e,*}\widehat{L}^{\vee}\otimes Q^{\vee})=0$ for any $i>0,e\gg 0$ which implies that 
\begin{align*}
& h^0(A,a_*(\Omega_X^1\otimes P)\otimes F^{e,*}\widehat{L}^{\vee}\otimes Q^{\vee}) \\
=& h^0(X,\Omega_X^1\otimes P\otimes a^*(F^{e,*}\widehat{L}^{\vee}\otimes Q^{\vee})) \\
=& h^2(X,\Omega_X^1\otimes P^{\vee}\otimes a^*(F^{e,*}\widehat{L}\otimes Q)) 
\end{align*}
is constant. To prove that 
\begin{align*}
& h^1(X,\Omega_X^1\otimes P^{\vee}\otimes a^*(F^{e,*}\widehat{L}\otimes Q)) \\
=& h^1(X,\Omega_X^1\otimes P\otimes a^*(F^{e,*}(\widehat{L}^{\vee})\otimes Q^{\vee}))
\end{align*} 
is constant we consider the Leray spectral sequence 
\begin{align*}
E_2^{ij}&=H^i(A,R^ja_*((\Omega_X^1\otimes P)\otimes a^*(F^{e,*}(\widehat{L}^{\vee})\otimes Q^{\vee}))) \\
&=H^i(A,R^ja_*(\Omega_X^1\otimes P)\otimes F^{e,*}(\widehat{L}^{\vee})\otimes Q^{\vee}) \\
&\Rightarrow H^{i+j}(X,\Omega_X^1\otimes P\otimes a^*(F^{e,*}(\widehat{L}^{\vee})\otimes Q^{\vee}))
\end{align*}
By Lemma \ref{2} we know $E_2^{ij}=0$ for $i>0$, so the spectral sequence degenerates at $E_2^{ij}$ and in particular  
$$H^1(X,\Omega_X^1\otimes P\otimes a^*(F^{e,*}(\widehat{L}^{\vee})\otimes Q^{\vee}))\cong H^0(A,R^1a_*(\Omega_X^1\otimes P)\otimes F^{e,*}(\widehat{L}^{\vee})\otimes Q^{\vee}).$$ 
Since $a$ is generically finite and $X$ is a surface, $R^1a_*(\Omega_X^1\otimes P)$ is supported on the locus whose preimage with respect to $a$ is $1$-dimensional, that is, a finite number of points. So we know $H^0(X,R^1a_*(\Omega_X^1\otimes P)\otimes F^{e,*}(\widehat{L}^{\vee})\otimes N)$ is constant for any line bundle $N$, in particular for any $Q\in {\rm Pic}^0(X)$.
\end{proof}
\begin{proof}[Proof of Theorem \ref{1}.] The rest of the proof follows \cite[Lemme 2.9]{DI87}. We fix an $e_0$ which satisfies the condition in Lemma \ref{3}, so after replacing $P$ by $F^{e,*}P$, by projection formula and for dimensional reasons we know that $H^j(X,F_{e,*}\Omega_X^i\otimes P\otimes a^*\widehat{L}^{\vee})=0$ for any $i+j\ge 3$, $e\ge e_0$ and $P\in {\rm Pic}^0(X)$. By Serre duality this is equivalent to $H^j(X,F_{e,*}\Omega_X^i\otimes P\otimes a^*\widehat{L})=0$ for any $i+j\le 1$ and $P\in{\rm Pic}^0(X)$. We consider the spectral sequence
$$E_1^{ij}=H^j(X,F_{e,*}\Omega_X^i\otimes P\otimes a^*\widehat{L})\Rightarrow H^{i+j}(X,F_{e,*}\Omega_X^\cdot\otimes P\otimes a^*\widehat{L}).$$
This implies that 
$$0=H^i(X,F_{e,*}\Omega_X^\cdot\otimes P\otimes a^*\widehat{L})=H^i(X,\tau_{<2}F_{e,*}\Omega_X^\cdot\otimes P\otimes a^*\widehat{L})$$
for $i\le 1$. Moreover since $X$ lifts to $W_2(k)$ and ${\rm char}(k)\ge 2>1$, by Theorem \ref{p12} we have in $D(X)$ an isomorphism 
$$\tau_{<2}F_*\Omega_X^{\cdot}\cong \bigoplus_{i<2}\Omega_X^i[-i].$$
Then
$$0=H^i(X,F_{e,*}\Omega_X^\cdot\otimes P\otimes a^*\widehat{L})=\bigoplus_l H^{i-l}(X,F_{e-1,*}\Omega_X^l\otimes P\otimes a^*\widehat{L})$$
for $i\le 1$. By descending induction on $e$ we know that $H^{i-l}(X,\Omega_X^l\otimes P\otimes a^*\widehat{L})=0$ for $i\le 1$. Hence by Serre duality we finally get $H^{j}(X,\Omega_X^i\otimes P\otimes a^*\widehat{L}^{\vee})=0$ for any $i+j\ge 3$ and $i\ge 0$. The second statement follows from Theorem \ref{p1} and Theorem \ref{p5}.
\end{proof}
Theorem \ref{4} and Corollary \ref{5} below describe the relationship between Kawamata-Viehweg vanishing and generic vanishing. 

\begin{theorem}\label{4} 
Let $a:X\to A$ be a generically finite morphism from a smooth projective variety $X$ to an abelian variety $A$, $\hat{A}$ the dual abelian variety of $A$ and $L$ an ample line bundle on $\hat{A}$. For an $m\in\mathbb{Z}^+$ let $\phi_{L^{\otimes m}}:\hat{A}\to A$ be the isogeny induced by $L^{\otimes m}$. Let $\hat{X}_m=X\times_A\hat{A}$ be the fiber product with respect to the morphisms $a$ and $\phi_{L^{\otimes m}}$, and let $\hat{a}_m:\hat{X}_m\to\hat{A}$ and $\varphi_m:\hat{X}_m\to X$ be the induced morphisms. \\
{\rm (a)} If $H^i(\hat{A},\hat{a}_{m,*}\omega_{\hat{X}_m}\otimes L^{\otimes m}\otimes P)=0$ for any $P\in {\rm Pic}^0(\hat{A})$ and $i>0$, then $H^j(A,a_*\omega_X\otimes\widehat{L^{\otimes m}}^{\vee})=0$ for any $j>0$. \\
{\rm (b)} If $H^i(A,a_*\omega_X\otimes\widehat{L^{\otimes m}}^{\vee})=0$ for any $i>0$ and $m\gg 0$, then for any ample line bundle $M$, $H^j(\hat{A},\hat{a}_{n,*}\omega_{\hat{X}_n}\otimes\phi_{L^{\otimes n}}^*(M))=0$ for any $j>0$ and $n\gg 0$.
\end{theorem} 
\begin{corollary}\label{5}
We use the notation in Theorem \ref{4} but assume that $X$ is a smooth projective surface. \\
{\rm (a)} If $H^i(\hat{X}_{m},\omega_{\hat{X}_m}\otimes \hat{a}_m^*(L^{\otimes m}\otimes P))=0$ for any $P\in {\rm Pic}^0(\hat{A})$ and $i>0$, then $H^j(A,a_*\omega_X\otimes\widehat{L^{\otimes m}}^{\vee})=0$ for any $j>0$. \\
{\rm (b)} If $H^i(A,a_*\omega_X\otimes\widehat{L^{\otimes m}}^{\vee})=0$ for any $i>0$ and $m\gg 0$ then for any ample line bundle $M$,  $H^j(\hat{X}_{n},\omega_{\hat{X}_n}\otimes{\hat{a}_n}^*\phi_{L^{\otimes n}}^*(M))=0$ for any $j>0$ and $n\gg 0$.
\end{corollary}
\begin{proof}
This is a direct consequence of Theorem \ref{4} and Theorem \ref{p1}.
\end{proof}
\begin{proof}[Proof of Theorem \ref{4}.] First we prove (a). We need the following lemma:
\begin{lemma}\label{9}
$\hat{a}_{m,*}\omega_{\hat{X}_m}=\phi^*_{L^{\otimes m}}a_*\omega_X$.
\end{lemma}
\begin{proof}
Since $\hat{a}_{m,*}\omega_{\hat{X}_m}$ is a sheaf, we can view it as a complex that is nonzero only in degree $0$, so by Theorem \ref{p1} and Grothendieck duality we have
\begin{align*}
\hat{a}_{m,*}\omega_{\hat{X}_m}=& R\hat{a}_{m,*}\omega_{\hat{X}_m}=R\hat{a}_{m,*}R\mathcal{H}om(\mathcal{O}_{\hat{X}_m},\omega_{\hat{X}_m}) \\
=&R\mathcal{H}om(R\hat{a}_{m,*}\mathcal{O}_{\hat{X}_m},\omega_{\hat{A}}).
\end{align*}
On the other hand since $A$ and $\hat{A}$ are abelian varieties we have $\omega_A=\mathcal{O}_A$ and $\omega_{\hat{A}}=\mathcal{O}_{\hat{A}}$. By \cite[Proposition 5.2]{MvdG}, $\phi_{L^{\otimes m}}$ is flat, so $\phi_{L^{\otimes m}}^*Ra_*\mathcal{O}_X= R\hat{a}_{m,*}\varphi_m^*\mathcal{O}_X$. Then 
\begin{align*}
\hat{a}_{m,*}\omega_{\hat{X}_m}=& R\mathcal{H}om(R\hat{a}_{m,*}\mathcal{O}_{\hat{X}_m},\omega_{\hat{A}})=R\mathcal{H}om(R\hat{a}_{m,*}\varphi_m^*\mathcal{O}_{X},\omega_{\hat{A}}) \\
=& R\mathcal{H}om(\phi_{L^{\otimes m}}^*Ra_*\mathcal{O}_{X},\phi_{L^{\otimes m}}^*\omega_{A})
=\phi_{L^{\otimes m}}^*R\mathcal{H}om(Ra_*\mathcal{O}_{X},\omega_{A}) \\
= &\phi_{L^{\otimes m}}^*Ra_*R\mathcal{H}om(\mathcal{O}_{X},\omega_{X})=\phi_{L^{\otimes m}}^*a_*\omega_X.
\end{align*}
The fourth equality is by \cite[Proposition 5.8]{Hartshorne66} and the fact that $\phi_{L^{\otimes m}}$ is flat, and the fifth equality is by Grothendieck duality.
\end{proof}

We will now prove that $H^j(A,a_*\omega_X\otimes\widehat{L^{\otimes m}}^{\vee})=0$ for any $ j>0$. By cohomology and base change it suffices to show that $R\hat{S}(a_*\omega_X\otimes\widehat{L^{\otimes m}}^{\vee})$ is a sheaf in degree $0$. Since $\hat{\phi}_{L^{\otimes m}}$ is finite, it is equivalent to show that     
\begin{align*}
&\hat{\phi}_{L^{\otimes m},*} R\hat{S}(a_*\omega_X\otimes\widehat{L^{\otimes m}}^{\vee})=RS(\phi_{L^{\otimes m}}^*(a_*\omega_X\otimes\widehat{L^{\otimes m}}^{\vee}))  \\
=& RS(\hat{a}_{m,*}\omega_{\hat{X}_m}\otimes\bigoplus^{h^0(L^{\otimes m})}(L^{\otimes m}))
\end{align*}  
is a sheaf in degree $0$, i.e. $\displaystyle R^jS(\hat{a}_{m,*}\omega_{\hat{X}_m}\otimes\bigoplus^{h^0(L^{\otimes m})}(L^{\otimes m}))=0, \forall j>0$. Here the first equality is by Lemma \ref{p11} and the second equality is by Lemma \ref{9}. Again by cohomology and base change it is equivalent to the vanishing of $H^j(\hat{A},\hat{a}_{m,*}\omega_{\hat{X}_m}\otimes L^{\otimes m}\otimes P)$ for any $P\in {\rm Pic}^0(\hat{A})$ and $j>0$, which is exactly the assumption. 

Now we prove part (b). We will need the following lemmas. 
\begin{lemma}\label{8}
For any integer ${m>0}$ we have $\phi_{L^{\otimes m}}=m_{A}\circ\phi_{L}$. Here $m_{A}$ is the multiplication by $m$ on $A$.
\end{lemma}
\begin{proof}
By definition of $m_A$ and \cite[Corollary 7.17]{MvdG} for any $P\in {\rm Pic}^0(\hat{A})=A$  we have $P^{\otimes m}=m_A(P)$. So
$$\phi_{L^{\otimes m}}(x)=(L^{\otimes m})^{\vee}\otimes t_x^*(L^{\otimes m})=(L^{\vee}\otimes t_x^*L)^{\otimes m}=m_{A}(L^{\vee}\otimes t_x^*L)=m_{A}\circ\phi_{L}(x),\ \forall x\in \hat{A}.$$ 
\end{proof}
\begin{lemma} \label{11}
For any ample line bundle $M$, $\phi_{L^{\otimes n}}^*(M)\otimes (L^{\otimes n})^{\vee}$ is ample for $n\gg 0$.
\end{lemma}
\begin{proof}
By Lemma \ref{8} we have
$$\phi_{L^{\otimes n}}^*M\otimes (L^{\otimes n})^{\vee}=\phi_L^*n_A^*M\otimes (L^{\otimes n})^{\vee}\equiv_{\rm num}(\phi_L^*M)^{\otimes n^2}\otimes (L^{\otimes n})^{\vee}=((\phi_L^*M)^{\otimes n}\otimes L^{\vee})^{\otimes n},$$
and $(\phi_L^*M)^{\otimes n}\otimes L^{\vee}$ is ample for $n\gg 0$.
\end{proof}
Suppose $H^i(A,a_*\omega_X\otimes\widehat{L^{\otimes n}}^{\vee})=0$ for some $m\gg 0$ and any $i>0$, then the proof of \cite[Proposition 4.1]{HK12} implies that $H^j(A,a_*\omega_X\otimes\widehat{L^{\otimes n}}^{\vee}\otimes P)=0$ for any $P\in {\rm Pic}^0(A)$ and $j>0$. Next we use the same trick as in the proof of (a) which implies that $R^jS(\phi_{L^{\otimes m}}^*(a_*\omega_X)\otimes(L^{\otimes n}))=0$, and this is again equivalent to $H^j(\hat{A},\hat{a}_{m,*}\omega_{\hat{X}_m}\otimes L^{\otimes n}\otimes P)=0$ by Lemma \ref{9} and cohomology and base change. Moreover since $\phi_{L^{\otimes n}}^*(M)\otimes (L^{\otimes n})^{\vee}$ is ample for $n\gg 0$, by \cite[Example 2.2 and Theorem 2.9]{PP03} we have 
$$H^i(\hat{A},\hat{a}_{m,*}\omega_{\hat{X}_m}\otimes\phi_{L^{\otimes n}}^*(M))=H^i(\hat{A},\hat{a}_{m,*}\omega_{\hat{X}_m}\otimes L^{\otimes n}\otimes(\phi_{L^{\otimes n}}^*(M)\otimes (L^{\otimes n})^{\vee}))=0.$$
\end{proof}
\section{Irrational fibration of surfaces with Euler characteristic 0}
In this section we assume that $k$ is of positive characteristic. We start this section by giving a lower bound for the Euler characteristic $\chi(\mathcal{O}_X)=\chi(\omega_X)$ of certain surfaces using the results in Section \ref{s3}.
\begin{proposition}\label{f1}
If $X$ is a smooth projective surface which is mAd and lifts to $W_2(k)$, then $\chi(\omega_X)\ge 0$.
\end{proposition}
\begin{proof}
Since $X$ lifts to $W_2(k)$, by Theorem \ref{1} and Theorem \ref{p1} we know $H^i(X,\omega_X\otimes P)=0$ for any $i>0$ and general $P\in {\rm Pic}^0(X)$. Then we take such a $P$ and by Riemann-Roch theorem we have $\chi(\omega_X)=\chi(\omega_X\otimes P)=h^0(X,\omega_X\otimes P)\ge 0$.
\end{proof}

Next we give some examples that satisfy the conditions in \ref{f1} and their Euler characteristics are equal to $0$.
\begin{example}
It is easy to see that the Euler characteristic of an abelian surface $A$ is $0$ and by \cite{MS87} Theorem 1 in Appendix, ordinary abelian varieties actually lift to $W(k)$, the ring of Witt vectors.
\end{example}
\begin{example}\label{f4}
Let $C_1$, $C_2$ be two smooth curves with $g(C_1)=1$ and $g(C_2)=n\ge 1$. Let $X=C_1\times C_2$ then obviously $X$ is mAd, and by the K\"unneth formula and Serre duality we can calculate $\chi(\omega_X)$ as follows:
\begin{align*}
\chi(\omega_X)=\chi(\omega_{C_1})\chi(\omega_{C_2})=(g(C_1)-1)(g(C_2)-1)=0 
\end{align*}
It is known that the only obstruction to lifting a smooth variety $X$ over $k$ of characteristic $p$ to $W_2(k)$ lies in ${\rm Ext}^2(\Omega_{X/k}^1,\mathcal{O}_X)$. When $X$ is a curve this group is $0$, which implies that every curve can be lifted to $W_2(k)$. Therefore $C_1\times C_2$ as above can also be lifted to $W_2(k)$. 
\end{example}
 
Example \ref{f4} has two obvious fibrations onto $C_1$ and $C_2$ and they are both irrational. Actually according to the Enriques-Kodaira classification for algebraic surfaces we know that for a smooth surface $X$ that lifts to $W_2(k)$, $\chi(\mathcal{O}_X)=0$ can only happen when $X$ is an elliptic surface, a quasi-elliptic surface or an abelian surface (see \cite{Liedtke13}). However in the case of elliptic surfaces or quasi-elliptic surfaces it is not clear what the base curve of the corresponding elliptic fibration looks like. The following theorem addresses this issue:
\begin{theorem}\label{f2}
Let $X$ be a smooth minimal projective surface which is mAd and lifts to $W_2(k)$, and its Picard variety has no supersingular factors. If $\chi(\omega_X)=0$ then either $X$ has an irrational pencil of genus  $\ge {\rm dim}(V^1(\omega_X))\ge 1$ or $X$ is an abelian surface. 
\end{theorem} 

To prove the theorem we need the following lemma which will also be useful in later sections.
\begin{lemma}\label{f3}
Let $X$ be a projective surface which is mAd and lifts to $W_2(k)$, and its Picard variety has no supersingular factors. Then either $X$ admits an irrational pencil of genus $\ge {\rm dim}(V^1(\omega_X))\ge 1$ or ${\rm dim} V^1(\omega_X)=0$.
\end{lemma}
\begin{proof}
Suppose that $X$ does not admit an irrational pencil of genus $\ge {\rm dim}(V^1(\omega_X))$ and ${\rm dim} V^1(\omega_X)\ge 1$. Let $a:X\to A$ be the Albanese morphism. By \cite[Corollary 3.4]{PR04} $V^1(\omega_X)$ is completely linear, so we can take $T+Q$ to be a component of maximal dimension of $V^1(\omega_X)$ where $T$ is abelian subvariety of dimension $\ge 1$ and $Q$ is a torsion element. Now $T$ is an abelian subvariety of $\hat{A}$, so after taking its dual we get a surjective morphism $c:A\to \hat{T}$. Denote $c\circ a$ by $g$. If ${\rm dim}(g(X))=1$ then as $g(X)$ generates $\hat{T}$, its genus must be $\ge {\rm dim}(T)={\rm dim}(V^1(\omega_X))\ge 1$. Since we have supposed this is not the case we must have ${\rm dim}(g(X))=2$. By Theorem \ref{1} and Theorem \ref{p1} we know that
$$h^1(X,\omega_X\otimes Q\otimes g^*R)=h^1(\hat{T},g_*(\omega_X\otimes Q)\otimes R)=0$$
 for general $R\in {\rm Pic}^0(\hat{T})=T$, which contradicts the definition of $T+Q$. 
\end{proof}
\begin{proof}[Proof of Theorem \ref{f2}.]
If $X$ does not have an irrational pencil of genus $\ge {\rm dim}V^1(\omega_X)\ge 1$ then by Lemma \ref{f3} ${\rm dim} V^1(a_*\omega_X)=0$, and this together with $\chi(\omega_X)=0$ implies that ${\rm dim} V^0(a_*\omega_X)=0$. Let $\mathcal{F}=R\hat{S}(D_A(a_*\omega_X))$. Then by Theorem \ref{1}, Theorem \ref{p5} and \cite[Corollary 3.5]{HP13} we see that $\mathcal{F}$ is a sheaf in degree $0$ supported on a finite number of points, hence by \cite[Example 2.9]{Mukai81} $V=(-1_A)^*(a_*\omega_X)=RS(D_{\hat{A}}(\mathcal{F}))$ is a vector bundle with a filtration
$$0=V^1\subset V^2\subset ...\subset V^n=V$$
where each $V^j/V^{j-1}$ is an element in ${\rm Pic}^0(A)$. In particular 
$${\rm dim}({\rm Supp}(a_*\omega_X))={\rm dim}({\rm Supp}(V))={\rm dim}(A)$$ 
which implies that ${\rm dim}(A)={\rm dim}(X)$ and hence $a:X\to A$ is surjective.

By Grothendieck-Riemann-Roch theorem we have
\begin{displaymath}
{\rm ch}(a_!\omega_X){\rm td}(T_A)=a_*({\rm ch}(\omega_X){\rm td}(T_X)),
\end{displaymath}
here $\displaystyle a_{!}\omega_X=\sum(-1)^iR^ia_*\omega_X$, and by Theorem \ref{p1} $R^ia_*\omega_X=0$ for $i>0$ so we have $a_!\omega_X=a_*\omega_X$. We compare the terms in $A^1(A)$ on both sides respectively. Note that $(-1_A)^*(a_*\omega_X)=V$ and as $V$ is of the above form we have 
$$c_1(a_*\omega_X)=\sum_{j}c_1(V_j/V_{j-1})\equiv_{\rm num}0.$$ 
We also know that $c_1(T_A)=-K_A=0$ and ${\rm td}(T_A)=1$, so it follows that the $A^1(A)$ part of the left side is numerically equivalent to $0$. Moreover the $A^1(A)$ part of the right side is 
$$a_*(c_1(\omega_X)+\frac{1}{2}c_1(T_X))=a_*(\frac{1}{2}K_X)\equiv_{\rm num}0.$$ 
By the mAd assumption on $X$ and the Enriques-Kodaira classification of surfaces we know $\kappa(X)\ge 0$, so there exists a divisor $D\ge 0$ such that $mK_X\sim D$ for some $m>0$. Next we pass to Stein factorization of $a$ as
$$X\xrightarrow{\tilde{a}}\tilde{A}\xrightarrow{b}A$$
where $\tilde{a}$ is birational and $b$ is finite. By minimality of $X$ we know that $D$ is nef over $\tilde{A}$. By the negativity lemma (cf. \cite[Lemma 3.39]{KM98}) we know $D\le 0$, hence $D=0$, thus $\kappa(X)=0$. Under our assumption of mAd and by the Enriques-Kodaira classification of surfaces this can only happen when $X$ is an abelian surface (cf. \cite[Section 7]{Liedtke13}).
\end{proof}

\section{On the classification of surfaces of general type with Euler characteristic 1}\label{s5}
We begin this section by giving an upper bound for genus and irregularity of surfaces with Euler characteristic $1$. We denote by $p_g(X)$ and $q(X)$ the geometric genus and the irregularity of $X$ respectively (see Definition \ref{Irreg&Euler}) and write $p_g$ and $q$ for short if no confusion can be made.
\begin{proposition}\label{c10}
Let $X$ be a minimal projective surface of general type over $k$ such that $p_g\ge 2$ and ${\rm char}(k)>0$. Then
\begin{align}\label{c11}
K_X^2\ge 2p_g+q-4.
\end{align}
If moreover $X$ is liftable to $W_2(k)$ and $\chi(\mathcal{O}_X)=1$ then $p_g=q\le 4$. 
\end{proposition}
\begin{proof}
First we would like to prove \eqref{c11}. \eqref{c11} is well known to experts, but we include a proof for the benifit of the reader. Since we assume $p_g\ge 2$ we can write $|K_X|=M+Z$, where $M$ is a linear series that has no fixed divisors and $Z$ is the fixed part. Note that members in $M$ are not necessarily smooth. We can take $D\in M$ and consider the exact sequence
\begin{align}\label{c7}
0\to\mathcal{O}_X\to\mathcal{O}_X(D)\to\mathcal{O}_D(D|_D)\to 0
\end{align} 
which yields a long exact sequence
\begin{align}\label{c16}
0\to H^0(\mathcal{O}_X)\to H^0(\mathcal{O}_X(D))\xrightarrow{r_1}H^0(\mathcal{O}_D(D|_D))\to ...
\end{align}
After twisting \eqref{c7} by $\mathcal{O}_X(D)$ and take the long exact sequence, we have
\begin{align}\label{c17}
0\to H^0(\mathcal{O}_X(D))\to H^0(\mathcal{O}_X(2D))\xrightarrow{r_2}H^0(\mathcal{O}_D(2D|_D))\to ...
\end{align} 
Now we prove that ${\rm dim}({\rm Im}(r_2))\ge 2{\rm dim}({\rm Im}(r_1))-1$. We have a homomorphism $r: {\rm Im}(r_1)\oplus {\rm Im}(r_1)\to {\rm Im}(r_2)$ induced by the commutativity of the following diagram   
\begin{center}
\begin{tikzpicture}[scale=1.6]
\node (A) at (0,1) {$H^0(\mathcal{O}_X(D))\otimes H^0(\mathcal{O}_X(D))$};
\node (B) at (4,1) {$H^0(\mathcal{O}_D(D|_D))\otimes H^0(\mathcal{O}_D(D|_D))$};
\node (C) at (0,0) {$H^0(\mathcal{O}_X(2D))$};
\node (D) at (4,0) {$H^0(\mathcal{O}_D(2D|_D))$};  
\path[->,font=\scriptsize]
(A) edge node[above]{$r_1\otimes r_1$} (B)
(A) edge node[right]{} (C)
(B) edge node[right]{} (D)
(C) edge node[above]{$r_2$} (D);
\end{tikzpicture} 
\end{center}    
When we view those images as linear systems the map $r_1$ is finite, because divisors on a curve are points, and there are only finitely many ways to separate these points into two parts. Therefore we have 
\begin{align}\label{c18}
{\rm dim}({\rm Im}(r_2))-1\ge 2({\rm dim}({\rm Im}(r_1))-1),
\end{align} 
which is exactly what we want. Next, by Riemann-Roch  
\begin{align}\label{c6}
h^0(\mathcal{O}_X(2K_X))=\chi(\mathcal{O}_X)+K_X^2=p_g-q+1+K_X^2.
\end{align} 
By \eqref{c16}, \eqref{c17} and \eqref{c18} we know that 
\begin{align*}
& h^0(\mathcal{O}_X(2K_X))-h^0(\mathcal{O}_X(K_X))= {\rm dim}({\rm Im}(r_2)) \\
\ge& 2{\rm dim}({\rm Im}(r_1))-1=2h^0(\mathcal{O}_X(K_X))-3.
\end{align*}
Then by \eqref{c6}
\begin{align*}
p_g-q+1+K_X^2= h^0(\mathcal{O}_X(2K_X))\ge 3h^0(\mathcal{O}_X(K_X))-3=3p_g-3.
\end{align*} 
Therefore we have proved \eqref{c11}. If $\chi(\mathcal{O}_X)=1$ and $X$ is liftable to $W_2(k)$ then by Proposition \ref{p6} we have $9\ge K_X^2\ge 2p_g+q-4$ and Proposition \ref{c10} follows easily.
\end{proof}
In this section we will consider surfaces that satisfy the following condition:
\begin{quotation}\label{star}
\item[(*)]$X$ is a projective surfaces that is mAd and lifts to $W_2(k)$, its Picard variety has no supersingular factors and $\chi(\mathcal{O}_X)=1$.
\end{quotation}
\bigskip
The main theorem of this section is as follows.
\begin{theorem}\label{c4}
Let $X$ be a smooth minimal projective surface of general type over an algebraically closed field $k$ of characteristic $\ge 11$ that satisfies {\rm (*)}. Denote the Albanese morphism as $a: X\to A$ and asssume that $a$ is separable. If ${\rm dim}(A)=4$ then $X=C_1\times C_2$ where $C_1$ and $C_2$ are smooth curves and $g(C_1)=g(C_2)=2$.
\end{theorem}
\begin{remark} \label{c12}
By the main result of \cite{Igusa55} we know ${\rm dim}(A)\le q(X)$, so by Proposition \ref{c10}, ${\rm dim}(A)\le 4$ and ${\rm dim}(A)=4$ implies $p_g(X)=q(X)=4$. 
\end{remark}

Next we will show that under the condition of the above theorem $X$ has at least two distinct fibrations onto curves of certain genera.
\begin{proposition}\label{c1}
Let $X$ be a smooth minimal projective surface over $k$ of positive characteristic that satisfies {\rm (*)}. If ${\rm dim}(A)=4$ then ${\rm dim}(V^1(\omega_X))\ge 1$. In particular $X$ admits an irrational pencil.
\end{proposition}
\begin{proof} 
Suppose that this is not the case. By Lemma \ref{f3} ${\rm dim}(V^1(\omega_X))=0$, so $h^0(X,\omega_X\otimes P)=1$ for all but finitely many $P\in {\rm Pic}^0(X)$. We also know that $a_*\omega_X$ is $GV_0$, so we have
\begin{align*}
& R^0\widehat{S}(D_A(a_*\omega_X))=R\widehat{S}(D_A(a_*\omega_X))\\
=& R\widehat{S}(D_A\circ Ra_*(\omega_X))=R\widehat{S}(Ra_*\circ D_X(\omega_X)) \\
=& R\widehat{S}(Ra_*R\mathcal{H}om(\omega_X, \omega_X[2]))=R\hat{S}Ra_*\mathcal{O}_X[2],
\end{align*}
where the first equality is by Theorem \ref{p5}, the second equality is by Theorem \ref{p1} and the third equality is by Grothendieck duality. This means that $R\hat{S}Ra_*\mathcal{O}_X$ is a sheaf in degree $2$. We claim that $R\hat{S}Ra_*\mathcal{O}_X[2]=L\otimes\mathcal{I}_Z$ where $L$ is a line bundle and $Z$ is a $0$-dimensional subvariety of $X$.

Next we prove the claim. We first prove that $R\hat{S}Ra_*\mathcal{O}_X[2]$ is torsion-free. Denote $R\hat{S}Ra_*\mathcal{O}_X[2]$ by $\mathcal{F}$. At the beginning of the proof we have deduced that $h^0(X,\omega_X\otimes P)=1$ for all but finitely many $P\in {\rm Pic}^0(X)$, and since $\chi(\mathcal{O}_X)=1$ we see that $h^i(X,\omega_X\otimes P)=0$ also for all but finitely many $P\in {\rm Pic}^0(X)$ and any $i>0$. So by cohomology and base change $\mathcal{F}$ is a line bundle except for a finite number of points. By Theorem \ref{p10} the following equality holds 
\begin{align}\label{c9}
(-1_A)^*Ra_*\mathcal{O}_X[-4]=RSR\hat{S}(Ra_*\mathcal{O}_X)=RS(\mathcal{F}[-2]),
\end{align}  
which means $R^0S(\mathcal{F})=R^1S(\mathcal{F})=0$ and $R^2S(\mathcal{F})=a_*\mathcal{O}_X$. On the other hand we have the following exact sequence
\begin{align}\label{c19}
0\to T\to \mathcal{F}\to \mathcal{F}^{\vee\vee}\to Q\to 0
\end{align}
where $T$ is supported on finitely many points. So when we consider the long exact sequence
\begin{displaymath}
0\to R^0S(T)\to R^0S(\mathcal{F})\to...
\end{displaymath}
we get $R^0S(T)=0$ because $R^0S(\mathcal{F})=0$ as above, and since $T$ is supported on finite many points we have $T=0$ and then $\mathcal{F}$ is torsion-free. Since $\mathcal{F}^{\vee\vee}$ is a reflexive sheaf of rank $1$ on a smooth variety, then $\mathcal{F}^{\vee\vee}=L$ is a line bundle, hence $\mathcal{F}=L\otimes\mathcal{I}_Z$ where $Z$ is a $0$-dimensional subscheme.

Now we consider the following short exact sequence:
\begin{displaymath}
0\to L\otimes\mathcal{I}_Z\to L\to L\otimes\mathcal{O}_Z\to 0,
\end{displaymath}
which yields a long exact sequence
\begin{align}\label{c8}
0\to R^0S(L\otimes\mathcal{I}_Z)\to R^0S(L)\to R^0S(L\otimes\mathcal{O}_Z)\to R^1S(L\otimes\mathcal{I}_Z)\to ... 
\end{align}
Among these terms $R^0S(L\otimes\mathcal{I}_Z)=R^1S(L\otimes\mathcal{I}_Z)=0$ and $R^2S(L\otimes\mathcal{I}_Z)=a_*\mathcal{O}_X$ as deduced above (immediately after \eqref{c9}), $R^iS(L\otimes\mathcal{O}_Z)=0$ for $i\ge 1$ because $L\otimes\mathcal{O}_Z$ is supported on a finite number of points, and $R^iS(L)\ne 0$ only for $i=i(L)$, the index of $L$. Denote $R^0S(L\otimes\mathcal{O}_Z)$ by $V$ which is a vector bundle in degree $0$. We get contradiction by considering the following two cases: \\
\emph{Case 1.} $Z=\emptyset$, which is equivalent to $V=0$. In this case $i(L)=2$ so $R^2S(L)= R^2S(L\otimes\mathcal{I}_Z)=a_*\mathcal{O}_X$. This means that the support of $R^2S(L)$ is $a(X)$. But on the other hand by Lemma \ref{p7} the support of $R^2S(L)$ must be an abelian subvariety, and since ${\rm dim}(a(X))=2$ we know that $a(X)$ does not generate $A$. Contradiction.  \\
\emph{Case 2.} $Z\ne\emptyset$, then $V$ is a nonzero vector bundle and $R^2S(L\otimes\mathcal{I}_Z)=a_*\mathcal{O}_X$ which is also nonzero. Thus $R^iS(L)\ne 0$ for $i=0$ and $i=2$ which is impossible as $R^j(L)=0$ for any $j\ne i(L)$.
\end{proof}
\begin{proposition}\label{c2}
Assume that $X$ satisfies all the conditions in Proposition \ref{c1}, then there are two irrational pencils on $X$ over two smooth curves $C_1$, $C_2$ satisfying either one of the following conditions 
\begin{enumerate}
\item both $g(C_1)$ and $g(C_2)$ are $\ge 2$,\label{condition(1)}
\item one of $g(C_1)$ and $g(C_2)$ is $\ge 3$ and the other is $1$, \label{condition(2)}
\end{enumerate}
such that the induced morphism $X\to C_1\times C_2$ is generically finite.
\end{proposition}
\begin{proof}
First by Remark \ref{codimV1} and Proposition \ref{c1} we have $1\le {\rm dim}(V^1(\omega_X))\le 3$. We take an irreducible component of maximal dimension in $V^1(\omega_X)$ and denote it as $Q_0+E$, where $Q_0$ is a torsion element and $E$ is an abelian subvariety. By Poincar\'{e}'s complete reducibility theorem we have an isogeny $E\times F\to\hat{A}$ where $F$ is an abelian subvariety of dimension ${\rm dim}(A)-{\rm dim}(V^1(\omega_X))=4-{\rm dim}(V^1(\omega_X))$. After dualizing this map we get the dual isogeny $b:A\to \hat{F}\times \hat{E}$. For convenience we denote $A_1=\hat{E}$, $A_2=\hat{F}$ and denote $pr_{\hat{E}}\circ b\circ a$ and $pr_{\hat{F}}\circ b\circ a$ by $a_1$ and $a_2$ respectively.

Next we prove that $X$ admits two dominant morphisms onto two smooth curves whose geometric genera satisfy \eqref{condition(1)} or \eqref{condition(2)} in Proposition \ref{c2}, and the induced morphism from $X$ to their product is generically finite. We consider the following three cases: \\
\textit{Case 1.} If ${\rm dim}(V^1(\omega_X))=1$, then ${\rm dim}(A_1)=1$ and ${\rm dim}(A_2)=3$. We then prove that $a_2$ induces an irrational pencil of genus $\ge 3$. Let 
$$\tilde{V^1}=\{(P,Q)\in \hat{A_1}\times \hat{A_2}|h^1(X,\omega_X\otimes a_1^*P\otimes a_2^*Q)\ne 0\},$$ 
then there is a finite map $\tilde{V^1}\to V^1=V^1(\omega_X)$ which implies that ${\rm dim}(\tilde{V^1})=1$. Let $P$ be a general element in ${\rm Pic}^0(A_1)$, and let 
$$S_P^1=\{Q\in {\rm Pic}^0(A_2)|(P,Q)\in \tilde{V^1}\}.$$ 
Then $S_P^1=\tilde{V^1}\cap(P\times \hat{A_2})$ is the fiber of the projection map $\tilde{V^1}\to \hat{A_1}$ over $P$. So we have 
$$1={\rm dim}(\tilde{V^1})= {\rm dim}(S_P^1)+{\rm dim}(\hat{A_1})={\rm dim}(S_P^1)+1$$ 
which forces ${\rm dim}(S_P^1)$ to be $0$. 

If $a_2$ is not generically finite then $a_2$ factors through $a(X)$ and the map $a(X)\to a_2(X)$ is an elliptic fibration onto its image. Since $a_2(X)$ generates $A_2$, $X$ is fibered over a curve of genus at least $3$. 

So next we would like to assume that $a_2$ is generically finite and then derive a contradiction. Since ${\rm dim}(S^1_P)=0$ we have that 
$$h^0(X,\omega_X\otimes a_1^*P\otimes a_2^*Q)=h^0(A_2,a_{2,*}(\omega_X\otimes a_1^*P)\otimes Q)=1$$ 
for all but finitely many $Q\in {\rm Pic}^0(A_2)$. By Theorem \ref{1} $a_{2,*}(\omega_X\otimes a_1^*P)$ is $GV_0$, so we have
\begin{align*}
& R^0\widehat{S}(D_{A_2}(a_{2,*}(\omega_X\otimes a_1^*P)))=R\widehat{S}(D_{A_2}(a_{2,*}(\omega_X\otimes a_1^*P)))\\
=& R\widehat{S}(D_{A_2}\circ Ra_{2,*}(\omega_X\otimes a_1^*P))=R\widehat{S}(Ra_{2,*}\circ D_X(\omega_X\otimes a_1^*P)) \\
=& R\widehat{S}(Ra_{2,*}R\mathcal{H}om(\omega_X\otimes a_1^*P, \omega_X[2]))=R\hat{S}Ra_{2,*}(a_1^*P^{\vee})[2],
\end{align*}
where the first equality is by Theorem \ref{p5}, the second equality is by Theorem \ref{p1} and the third equality is by Grothendieck duality. This means $R\hat{S}Ra_{2,*}(a_1^*P^{\vee})$ is a sheaf in degree $2$. If we denote $R\hat{S}Ra_{2,*}(a_1^*P^{\vee})[2]$ by $\mathcal{G}$, then by Theorem \ref{p10} we know 
$$(-1_{\hat{A_2}})^*Ra_{2,*}(a_1^*P^{\vee})[-3]=RS(R\hat{S}Ra_{2,*}(a_1^*P^{\vee}))=RS(\mathcal{G}[-2]),$$
which means $R^0S(\mathcal{G})=0$ and $R^1S(\mathcal{G})=a_{2,*}(a_1^*P^{\vee})$. Then arguing as in the proof of Proposition \ref{c1} starting from \eqref{c19} with $\mathcal{F}$ replaced by $\mathcal{G}$ we see that $\mathcal{G}=L\otimes \mathcal{I}_Z$ where $L$ is a line bundle and $Z$ is supported on a finite set. 

Now we can construct a long exact sequence as \eqref{c8}. In the long exact sequence we have $R^0S(L\otimes\mathcal{I}_Z)=0$, $R^1S(L\otimes\mathcal{I}_Z)=a_{2,*}(a_1^*P^{\vee})$, $R^iS(L)\ne 0$ only for $i=i(L)$ and $R^iS(L\otimes\mathcal{O}_Z)=0$ for $i\ge 1$.

Next we deduce contradiction for all $i(L)$. Since $L$ is supported on $\hat{A_2}$, by Proposition \ref{p7} we see that $i(L)\le {\rm dim}(\hat{A_2})=3$. So we consider $i(L)=3,2,1,0$ respectively.

If $i(L)=2$ or $3$ we can get $V=R^1S(L\otimes\mathcal{I}_Z)=a_{2,*}(a_1^*P^{\vee})$ where $V:=R^0S(L\otimes\mathcal{O}_Z)$. But since $V$ is a vector bundle and $a_{2,*}(a_1^*P^{\vee})$ is a torsion sheaf this is a contradiction. 

If $i(L)=1$, we have a short exact sequence
\begin{displaymath}
0\to V\to a_{2,*}(a_1^*P^{\vee})\to \widehat{L}\to 0
\end{displaymath}
which forces $V=0$ because $a_{2,*}(a_1^*P^{\vee})$ is a torsion sheaf. Then the support of $\widehat{L}$ is $a_2(X)$. We already know from Proposition \ref{p7} that the support of $\widehat{L}$ is an abelian subvariety, and on the other hand according to the above construction of $a_2$ we know $a_2(X)$ is $2$-dimensional and generates $A_2$. This is a contradiction. 

If $i(L)=0$ then we have a short exact sequence
\begin{align}\label{sesV}
0\to \widehat{L}\to V\to a_{2,*}(a_1^*P^{\vee})\to 0.
\end{align}
After taking its dual, we get a long exact sequence
\begin{align}\label{c20}
0& \to a_{2,*}(a_1^*P^{\vee})^{\vee}\to V^{\vee}\to \widehat{L}^{\vee} \\
& \to \mathcal{E}xt^1(a_{2,*}(a_1^*P^{\vee}),\mathcal{O}_{A_2})\to \nonumber \mathcal{E}xt^1(V,\mathcal{O}_{A_2})\to ...
\end{align}
In \eqref{c20} we have $\mathcal{E}xt^1(V,\mathcal{O}_{A_2})=0$ as $V$ is locally free and $a_{2,*}(a_1^*P^{\vee})^{\vee}=0$ as $a_{2,*}(a_1^*P^{\vee})$ is a torsion sheaf. Then \eqref{c20} reduces to a short exact sequence
\begin{align}\label{c15}
0\to V^{\vee}\to \widehat{L}^{\vee}\to \mathcal{E}xt^1(a_{2,*}(a_1^*P^{\vee}),\mathcal{O}_{A_2})\to 0.
\end{align}
Since $\widehat{L}\ne 0$, by \eqref{sesV} we see that $V\ne 0$. Since $L\otimes \mathcal{O}_Z$ is supported on points, by \cite[Example 2.9]{Mukai81} we know $V^{\vee}=\bigoplus_iV_i$ where $V_i$ is a successive extension by elements in ${\rm Pic}^0(A_2)$. If we consider one step of such extension as the exact sequence
\begin{displaymath}
0\to W'\to W\to R\to 0
\end{displaymath}
where $R\in {\rm Pic}^0(A_2)$, then after twisting it by $R^{\vee}$ it becomes 
\begin{displaymath}
0\to W'\otimes R^{\vee}\to W\otimes R^{\vee}\to \mathcal{O}_{A_2}\to 0.
\end{displaymath} 
Then $H^3(A_2,\mathcal{O}_{A_2})\ne 0$ implies that $H^3(A_2,W\otimes R^{\vee})\ne 0$. Following such successive extension we finally get that there exists $P'\in {\rm Pic}^0(A_2)$ such that 
\begin{align}\label{21}
H^3(A_2,V^{\vee}\otimes P')\ne 0. 
\end{align}
Moreover by Theorem \ref{p10} we have
$$R\hat{S}RS(L)=(-1_{\hat{A_2}})^*L[-3],$$
then 
$$R^i\hat{S}(\widehat{L})=0,\forall i\ne 3.$$ 
So by cohomology and base change 
$$H^i(A_2,\widehat{L}\otimes Q)=0,\forall i\ne 3,Q\in {\rm Pic}^0(A_2),$$
in particular
$$H^3(A_2,\widehat{L}^{\vee}\otimes P')=H^0(A_2,\widehat{L}\otimes P'^{\vee})=0.$$ 
This together with \eqref{21} and \eqref{c15} implies that 
\begin{align}\label{Extnonzero}
H^2(A_2,\mathcal{E}xt^1(a_{2,*}(a_1^*P^{\vee}),\mathcal{O}_{A_2})\otimes P')\ne 0.
\end{align}

On the other hand by Grothendieck duality we have
\begin{align*}
R\mathcal{H}om(Ra_{2,*}(a_1^*P^{\vee}),\mathcal{O}_{A_2}[3])=Ra_{2,*}(R\mathcal{H}om(a_1^*P^{\vee},\omega_X[2])),
\end{align*}
hence 
\begin{align}\label{GD}
R\mathcal{H}om(Ra_{2,*}(a_1^*P^{\vee}),\mathcal{O}_{A_2})[1]=Ra_{2,*}(R\mathcal{H}om(a_1^*P^{\vee},\omega_X)),
\end{align}
By Theorem \ref{p1} the right side of \eqref{GD} is just $a_{2,*}(\omega_X\otimes a_1^*P)$, so after taking cohomology of \eqref{GD} in degree $0$ we have 
\begin{align}\label{ExtDerived}
\mathcal{E}xt^1(Ra_{2,*}(a_1^*P^{\vee}),\mathcal{O}_{A_2})=a_{2,*}(\omega_X\otimes a_1^*P).
\end{align}
Now by \cite[(3.7)]{Huybrechts06} we have the following spectral sequence 
\begin{align}\label{SpectralSequenceExt}
E_2^{p,q}:=\mathcal{E}xt^p(R^{-q}a_{2,*}(a_1^*P^{\vee}),\mathcal{O}_{A_2})\Rightarrow \mathcal{E}xt^{p+q}(Ra_{2,*}(a_1^*P^{\vee}),\mathcal{O}_{A_2}).
\end{align}
Since $a_2$ is generically finite onto its image we see that $R^{-q}a_{2,*}(a_1^*P^{\vee})$ is $0$ for $q\le -2$ and $q\ge 1$, and $R^1a_{2,*}(a_1^*P^{\vee})$ is supported on points. Hence $\mathcal{E}xt^p(R^{-q}a_{2,*}(a_1^*P^{\vee}),\mathcal{O}_{A_2})$ is $0$ when $q=-1$ and $p\ne 3$, and $\mathcal{E}xt^3(R^{1}a_{2,*}(a_1^*P^{\vee}),\mathcal{O}_{A_2})$ is supported on points. Therefore the spectral sequence \eqref{SpectralSequenceExt} degenerates at $E_3^{p,q}$ and $\mathcal{E}xt^1(Ra_{2,*}(a_1^*P^{\vee}),\mathcal{O}_{A_2})$ is equal to the kernel of the following natural morphism in the second page
$$d_2^{1,0}: E_2^{1,0}=\mathcal{E}xt^1(a_{2,*}(a_1^*P^{\vee}),\mathcal{O}_{A_2})\to E_2^{3,-1}=\mathcal{E}xt^3(R^1a_{2,*}(a_1^*P^{\vee}),\mathcal{O}_{A_2}).$$
So by \eqref{ExtDerived} we have the following short exact sequence
\begin{align}\label{sesext}
0\to a_{2,*}(\omega_X\otimes a_1^*P)\to \mathcal{E}xt^1(a_{2,*}(a_1^*P^{\vee}),\mathcal{O}_{A_2})\to\mathcal{E}xt^3(R^1a_{2,*}(a_1^*P^{\vee}),\mathcal{O}_{A_2})\to 0.
\end{align} 
Twisting \eqref{sesext} by $P'$ constructed in \eqref{21} we have
\begin{align}\label{sesexttwist}
\begin{split}
0\to a_{2,*}(\omega_X\otimes a_1^*P)\otimes P' & \to \mathcal{E}xt^1(a_{2,*}(a_1^*P^{\vee}),\mathcal{O}_{A_2})\otimes P' \\
& \to\mathcal{E}xt^3(R^1a_{2,*}(a_1^*P^{\vee}),\mathcal{O}_{A_2})\otimes P'\to 0.
\end{split}
\end{align} 

By \eqref{Extnonzero} and the fact that $\mathcal{E}xt^3(R^1a_{2,*}(a_1^*P^{\vee}),\mathcal{O}_{A_2})\otimes P'$ is supported on points we know that $H^2(A_2,a_{2,*}(\omega_X\otimes a_1^*P)\otimes P')\ne 0$. On the other hand 
$$H^2(A_2,a_{2,*}(\omega_X\otimes a_1^*P)\otimes P')=H^2(X,\omega_X\otimes a_1^*P\otimes a_2^*( P'))=H^0(X,a_1^*P^{\vee}\otimes a_2^*(P'^{\vee})),$$ 
but since $P^{\vee}$ is a general element of $A_1$ we have $a_1^*P^{\vee}\otimes a_2^*P'^{\vee}\ne \mathcal{O}_X$, so $H^0(X,a_1^*P^{\vee}\otimes a_2^*(P'^{\vee}))=0$ and we get a contradiction. 

So we have deduced that $a_2$ induces an irrational pencil of genus $\ge 3$. On the other hand since $A_1$ is an elliptic curve $a_1$ must be surjective, otherwise $a(X)$ would be contained in a $3$-dimensional abelian subvariety of $A$ which is impossible. Therefore in this case $a_1$ and $a_2$ induce two dominant morphims we want.\\
\textit{Case 2.} If ${\rm dim}(V^1(\omega_X))=3$, then ${\rm dim}(A_1)=3$ and ${\rm dim}(A_2)=1$. By the proof of Lemma \ref{f3}, $a_1$ induces an irrational pencil of genus $\ge 3$ and by the argument at the end of \textit{Case 1}, $a_2$ induces a dominant morphism onto a curve of geometric genus $\ge 1$. So in this case we also have the two dominant morphisms we claimed. \\
\textit{Case 3.} If ${\rm dim}(V^1(\omega_X))=2$ then ${\rm dim}(A_1)={\rm dim}(A_2)=2$. We define $\tilde{V^1}$ in the same way as in \textit{Case 1} and define 
$$T_Q^1=\{P\in {\rm Pic}^0(A_1)|(P,Q)\in \tilde{V^1}\}$$ 
for a general $Q\in \hat{A_2}$. 

Next we prove that $T_Q^1\ne\emptyset$.  If $T_Q^1=\emptyset$, then
\begin{align}\label{hi=0}
h^i(X,\omega_X\otimes a_2^*Q\otimes a_1^*P)=0,\forall P\in {\rm Pic}^0(A_1), i>0,
\end{align}
so 
$$h^0(A_1,a_{1,*}(\omega_X\otimes a_2^*Q)\otimes P)=h^0(X,\omega_X\otimes a_2^*Q\otimes a_1^*P)=1,\forall P\in {\rm Pic}^0(A_1).$$ 
Consider the Leray spectral sequence 
$$E_2^{ij}=H^i(A_1,R^ja_{1,*}(\omega_X\otimes a_2^*Q)\otimes P)\Rightarrow H^{i+j}(X,\omega_X\otimes a_2^*Q\otimes a_1^*P).$$
We actually have $E_2^{ij}=0$ for $i\ge 2$ or $j\ge 2$. Indeed if $a_1$ is generically finite this is by \eqref{hi=0} and Theorem \ref{p1} and if $a_1$ is not generically finite this is for dimensional reasons. So this spectral sequence degenerates at $E_2$, in particular 
\begin{align*}
& h^1(X,\omega_X\otimes a_2^*Q\otimes a_1^*P) \\
=& h^1(A_1,a_{1,*}(\omega_X\otimes a_2^*Q)\otimes P)+h^0(A_1,R^1a_{1,*}(\omega_X\otimes a_2^*Q)\otimes P) \\
\ge & h^1(A_1,a_{1,*}(\omega_X\otimes a_2^*Q)\otimes P).
\end{align*}
Hence $h^1(A_1,a_{1,*}(\omega_X\otimes a_2^*Q)\otimes P)=0$. So by cohomology and base change $R\hat{S}_{A_1}(a_{1,*}(\omega_X\otimes a_2^*Q))$ is a line bundle in degree $0$ which we denote by $L_1$. By Theorem \ref{p10}, $a_{1,*}(\omega_X\otimes a_2^*Q)=(-1_{A_1})^*\widehat{L_1}$. However the support of $a_{1,*}(\omega_X\otimes a_2^*Q)$ is a curve which spans $A_1$ while by Proposition \ref{p7} the support of $(-1_{A_1})^*\widehat{L_1}$ is an abelian subvariety in $A_1$. This is a contradiction.  

We also have 
$$2={\rm dim}(\tilde{V^1})= {\rm dim}(T_Q^1)+{\rm dim}(\hat{A_2})={\rm dim}(T_Q^1)+2$$ 
which forces ${\rm dim}(T_Q^1)$ to be $0$. This, together with $T_Q^1\ne \emptyset$, implies that in $\tilde{V^1}\subset\hat{A_1}\times\hat{A_2}$ there is a $2$-dimensional component which is a torsion translate of an abelian subvariety and dominates $\hat{A_2}$. We denote this component by $Q_0+\hat{B_1}$, then $\hat{B_1}\cap (\hat{A_1}\times\{0\})$ is a finite number of points. This means that the natural homomorphism $\hat{A_1}\times\hat{B_1}\to \hat{A_1}\times\hat{A_2}$ is an isogeny (cf. proof of \cite[p.160 Theorem 1]{Mumford12}). We denote the following composition of isogenies
$$\hat{A_1}\times\hat{B_1}\to \hat{A_1}\times \hat{A_2}\to\hat{A}$$ 
by $\hat{\varphi}$. Then we have the dual isogeny $\varphi: A\to A_1\times B_1$. By the proof of Lemma \ref{f3} each of the two morphisms $X\to A_1$ and $X\to B_1$ gives a dominant morphism to a curve of geometric genus $\ge 2$. So in this case what we claimed also holds.

Therefore by the argument for the above three cases we have constructed an isogeny $A\to A_1\times A_2$ where $({\rm dim}(A_1), {\rm dim}(A_2))=(3,1)$ or $(2,2)$ (by symmetry we can assume that ${\rm dim}(A_1)\ge {\rm dim}(A_2)$). Each of the projections $\psi_i: A\to A_i$ induces a morphism $g_i: X\to a(X)\to \tilde{C_i}$ where $\tilde{C_i}$ is a smooth curve (we can pass to their normalization if necessary because their geometric genera do not change), and the genera of the two curves $(g(\tilde{C_1}),g(\tilde{C_2}))$ can be either $(m,1)$ or $(n,k)$ where $m\ge 3$ and $n,k\ge 2$. However so far $g_i$ may not satisfy $g_{i,*}\mathcal{O}_X=\mathcal{O}_{\tilde{C_i}}$. Denote the kernel of $\psi_i$ by $K_i$ and the connected component containing the origin by $K_i^0$. By \cite[Proposition 5.31]{MvdG} $(K_i^0)_{\rm red}$ is an abelian subvariety of $A$, so the quotient $A/(K_i^0)_{\rm red}$ exists and the quotient morphism $A\to A/(K_i^0)_{\rm red}$ is separable. For convenience we also use $A_i$ to denote $A/(K_i^0)_{\rm red}$ and use $\varphi_i$ and $h_i:X\to C_i$ to denote the quotient morphism $A\to A/(K_i^0)_{\rm red}$ and the Stein factorization of $\varphi_i\circ a$ respectively, and by passing to normalization we can assume that each $C_i$ is smooth. By construction we have that $h_{i,*}\mathcal{O}_X=\mathcal{O}_{C_i}$ and $g_i$ factors through $h_i$, so by the Hurwitz formula, the genera of $(g(C_1),g(C_2))$ are also either $(m,1)$ or $(n,k)$ where $m\ge 3$ and $n,k\ge 2$. So $h_1$ and $h_2$ are the two irrational pencils we want.
\end{proof}
\begin{proof}[Proof of Theorem \ref{c4}]
Let $h_i:X\to C_i$ be the two irrational pencils constructed above. WLOG we assume that $g(C_1)\ge g(C_2)$, in particular $g(C_1)\ge 2$. We have the following lemma:
\begin{lemma}\label{c5}
There is an injective morphism of sheaves 
\begin{align*}
\omega_{C_i}\otimes h_{i,*}\omega_X\to h_{i,*}\omega_X^2,
\end{align*}
where $i=1,2$.
\end{lemma}
\begin{proof}
By the above construction we have the following diagram:
\begin{center}
   \begin{tikzpicture}[scale=1.6]
   \node (A) at (0,1) {$X$};
   \node (B) at (2,1) {$a(X)$};
   \node (C) at (4,1) {$A$};
   \node (D) at (2,0) {$C_i$}; 
   \node (E) at (4,0) {$A_i$};    
   \path[->,font=\scriptsize]
   (A) edge node[above]{} (B)   
   (A) edge node[above]{$h_i$} (D)
   (C) edge node[right]{$\varphi_i$} (E);
   \path[->,font=\scriptsize]
   (B) edge node[above]{} (C)
   (D) edge node[above]{$j_i$} (E);
   \end{tikzpicture} 
\end{center}
Now by the separability of $\varphi_i$ we get following injective homomorphisms induced by pullbacks
\begin{align}\label{c13}
H^0(A_i,\Omega_{A_i}^1)\to H^0(A,\Omega_A^1).
\end{align}
Since $H^0(A_i,\Omega_{A_i}^1)\times \mathcal{O}_{A_i}\cong \Omega_{A_i}^1=\mathcal{T}_{A_i}^{\vee}$, we have $H^0(A_i,\Omega_{A_i}^1)\cong T_{A_i,e}^{\vee}$. Then we have 
\begin{align}\label{c14}
H^0(A,\Omega_A^1)\cong H^0(A_1,\Omega_{A_1}^1)\oplus H^0(A_2,\Omega_{A_2}^1)
\end{align} 
and the isomorphism is induced by pullback via $\varphi_1\times\varphi_2$. Then, after wedging $H^0(A,\Omega_A^1)$ with itself, by \eqref{c14} we get
\begin{align*}
& H^0(A,\Omega_A^2)=\wedge^2 H^0(A,\Omega_A^1) \\
=& \wedge^2H^0(A_1,\Omega_{A_1}^1)\oplus\wedge^2 H^0(A_2,\Omega_{A_2}^1)\oplus (H^0(A_2,\Omega_{A_2}^1)\otimes H^0(A_1,\Omega_{A_1}^1)). 
\end{align*}

By separability of $a$ (note that this is the only place where we use separability of $a$) there exists $\omega\in H^0(A,\Omega_A^2)$ such that $0\ne a^*\omega\in H^0(X,\Omega_X^2)$ on $X$. If $\omega\in \wedge^2H^0(A_i,\Omega_{A_i}^1)$ then since the above diagram commutes if we go through the pullback by $j_i\circ h_i$ then $\omega$ must go to $0$ as $C_i$ is $1$-dimensional. Therefore there are two $1$-forms $\omega_i\in H^0(A_i,\Omega_{A_i}^1)$ such that $a^*(\varphi_1\times\varphi_2)^*(\omega_1\boxtimes\omega_2)\ne 0$. This also implies that 
\begin{align*}
H^0(C_i,\omega_{C_i})\otimes H^0(A_j,\Omega_{A_j}^1)\to H^0(X,\omega_X)
\end{align*} 
is nonzero, hence
\begin{align}\label{c22}
h_i^*\omega_{C_i}\otimes H^0(A_j,\Omega_{A_j}^1)\to \omega_X
\end{align} 
induced by pullbacks is nonzero for $(i,j)=(1,2)$ or $(2,1)$. This means that there exists a $\omega_j'\in H^0(A_j,\Omega_{A_j}^1)$ such that the morphism 
\begin{align}\label{c23}
h_i^*\omega_{C_i}\xrightarrow{\wedge\omega_j'} \omega_X
\end{align}
is nonzero, hence by Lemma \ref{p2} it is injective.

Finally after twisting \eqref{c23} by $\omega_X$ and pushing it forward by $h_i$ and using the projection formula we are done. 
\end{proof}
Next we would like to prove that under the condition of Theorem \ref{c4} a general fiber $F$ of $h_1$ is smooth and to do this we need to estimate $p_a(F)$. By the adjunction formula we have 
$$\omega_F=(\omega_X+F)|_F=\omega_X|_F,$$ 
where $\omega_F$ is the dualizing sheaf of $F$. So by \cite[Theorem III.12.8 and Corollary III.12.9]{Hartshorne77} and Serre duality
$${\rm rk}(h_{1,*}\omega_X)=h^0(F,\omega_X|_F)=h^0(F,\omega_F)=h^1(F,\mathcal{O}_F).$$
Then by Riemann-Roch theorem on curves we have
\begin{align}
&\chi(\omega_{C_1}\otimes h_{1,*}\omega_X)=\chi(h_{1,*}\omega_X)+{\rm rk}(h_{1,*}\omega_X)(2g(C_1)-2) \nonumber\\
=& \chi(h_{1,*}\omega_X)+p_a(F)(2g(C_1)-2)\ge \chi(h_{1,*}\omega_X)+2p_a(F).
\label{eq} 
\end{align}
Next we make two observations: \\
1. We estimate the left hand side as follows:
\begin{align*}
& \chi(\omega_{C_1}\otimes h_{1,*}\omega_X)\le h^0(C_1,\omega_{C_1}\otimes h_{1,*}\omega_X)\le h^0(C_1,h_{1,*}(\omega_X^2))\\
=& h^0(X,\omega_X^2)=\chi(\mathcal{O}_X)+K_X^2\le 1+9=10
\end{align*}
where the second and the third inequalities are by Lemma \ref{c5} and Proposition \ref{p6} respectively. \\
2. We claim that $\chi(h_{1,*}\omega_X)\ge 0$. To prove this consider Leray spectral sequence $E_2^{p,q}=H^p(C_1, R^qh_{1,*}\omega_X)\Rightarrow H^{p+q}(X, \omega_X)$. Since $E_2^{p,q}=0$ for $p\ge 2$ for dimensional reasons we know that the spectral sequence degenerates at $E_2$. That means 
\begin{align*}
4=h^1(X,\omega_X)=h^0(C_1,R^1h_{1,*}\omega_X)+h^1(C_1, h_{1,*}\omega_X)\ge h^1(C_1,h_{1,*}\omega_X).
\end{align*} Finally we have $\chi(h_{1,*}\omega_X)=h^0(h_{1,*}\omega_X)-h^1(h_{1,*}\omega_X)\ge 4-4=0$.

Therefore in \eqref{eq} we know $10\ge 2p_a(F)$, so $p_a(F)\le 5$. Then Tate's theorem (cf. \cite[Theorem 5.1]{Liedtke13}) implies that if $F$ is singular then $(p-1)/2$ divides $p_a(F)-p_a(\tilde{F})$, in particular $(p-1)/2\le 5$. So when $p\ge 13$ $F$ is smooth. When $p=11$ and $F$ is singular then it can only happen that $p_a(F)=5$ and $p_a(\tilde{F})=0$. But then $F$ is rational which is contradictory to the assumption that $X$ is of mAd. So when $p=11$, $F$ is also smooth.

Since $X$ is mAd and of general type, by Proposition \ref{p8} $g(F)$ must be no less than $2$. So by a theorem of Arakelov (see \cite[Th\'{e}or\`{e}m d'Arakelov and Corollaire]{Beauville82}) we know
\begin{align*}
9\ge K_X^2\ge 8(g(C_1)-1)(g(F)-1)\ge 8(2-1)(2-1)=8.
\end{align*} 
This forces that $g(C_1)=g(F)=2$, so we see that the case $g(C_1)\ge 3$ cannot happen, in particular $g(C_1)=g(C_2)=2$. Then after we divide the morphism $F\to C_2$ into a separable one and a purely inseparable one, by Hurwitz's formula we have $F\cong C_2$ (cf. \cite[Ch.IV Example 2.5.4 and Proposition 2.5]{Hartshorne77}), and when we restrict $h_2$ to any fiber of $h_1$, say $X_t$ for $t\in C_1$ it is a composition of Frobenius morphisms $F^{e_t}$. Since for any $p\in C_2$ we get that $h_2^*\mathcal{O}_{C_2}(p)\cdot X_t$ is constant with respect to $t\in C_1$. So $e_t=e_1$ is constant with respect to $t$. We also do all the above argument for $h_2:X\to C_2$ and get that the general fiber of $h_2$ is isomorphic to $C_1$ and $h_1$ induces $F^{e_2}$ on these fibers for a uniform $e_2$. 

Denote the induced morphism $X\to C_1\times C_2$ by $f$. There is a generically \'{e}tale morphism $\phi:C_1'\to C_1$ such that $X\times_{C_1}C_1'=C_1'\times C_2$ and the induced map $X\times_{C_1}C_1'\to C_1'$ is the projection onto $C_1'$ (cf. \cite{Jang08} right after Definition 2.3), where $C_1$ is a projective curve. We denote the induced map $X\times_{C_1}C_1'=C_1'\times C_2\to X$ by $\varphi$ and the projections onto $C_1'$ and $C_2$ by $p_1$ and $p_2$ respectively.

%By \cite{Tziolas14} Proposition 4.1, Corollary 4.2 and the fact that $C_2$ lifts to $W_2(k)$ we know that ${\rm Aut}(C_2)$ is smooth, in particular reduced. Since ${\rm Aut}(C_2)$ is also finite, $G={\rm Im}(\rho)$ is reduced, so $\varphi$ and $\phi$ are separable. 

Now by the above construction $\varphi$ and $\phi$ are both separable morphisms and ${\rm deg}(\varphi)={\rm deg}(\phi)$. By construction of $h_2$ we know that $h_2\circ \varphi$ is also separable and contracts every $C_1'\times\{c\}$ where $c$ is a general closed point of $C_2$.  So by the Hurwitz formula $h_2\circ \varphi$ restricted to every $\{c'\}\times C_2$, where $c'$ is a general closed point of $C_1'$, is a composition of Frobenius morphisms and an automorphism of $C_2$, in particular on the underlying topological space it is an automorphism of $C_2$. Since ${\rm Aut}(C_2)$ is finite there is a uniform automorphism $\alpha: C_2\to C_2$ such that on the underlying topological spaces, $h_2\circ \varphi=\alpha\circ p_2$. Now we do a Stein factorization of $h_2\circ \varphi$ which we denote by 
$$C_1'\times C_2\xrightarrow{r}C_2\xrightarrow{\beta}C_2.$$ 
The situation is as follows.
\begin{center}
\begin{tikzpicture}[scale=1.6]
\node (A) at (0,1) {$X$};
\node (B) at (2,1) {$C_1\times C_2$};
\node (C) at (-2,0) {$C_1'$};
\node (D) at (2,0) {$C_1$}; 
\node (E) at (-2,1) {$C_1'\times C_2$}; 
\node (F) at (2,2) {$C_2$}; 
\node (G) at (0,2) {$C_2$}; 
\path[->,font=\scriptsize]
(A) edge node[above]{$f$} (B)
(E) edge node[right]{$p_1$} (C)
(B) edge node[right]{} (D)
(A) edge node[above]{$h_1$} (D)
(E) edge node[above]{$r$} (G)
(G) edge node[above]{$\beta$} (F)
(A) edge node[above]{$h_2$} (F)
(E) edge node[above]{$\varphi$} (A)
(C) edge node[above]{$\phi$} (D)
(B) edge node[right]{} (F);
\end{tikzpicture} 
\end{center}

Since $h_2\circ \varphi$ is separable we know that $\beta$ is separable and hence $\beta\in{\rm Aut}(C_2)$. By definition of Stein factorization we have $(\beta\circ r)_*\mathcal{O}_{C_1'\times C_2}=\mathcal{O}_{C_2}$, and we have seen above that the map of the underlying topological spaces for $\beta\circ r$ is $\alpha\circ p_2$. So $h_2\circ \varphi=\alpha\circ p_2$ as a morphism of schemes. On the other hand $\varphi\circ h_1=\phi\circ p_1$, so $f\circ \varphi=(\phi\circ p_1)\times (\alpha\circ p_2)$ and in particular it is separable with degree ${\rm deg}(\phi)$, which is equal to the degree of $\varphi$. This implies that the degree of $f$ is $1$, hence $f$ is birational. Finally by smoothness and minimality of $X$ we have $X=C_1\times C_2$.
\end{proof}
Proposition \ref{c1} has told us that if $X$ satisfies (*) then ${\rm dim}(V^1(\omega_X))=0$ can only happen when ${\rm dim}(A)\le 3$. We end this section with a result for such case.
\begin{proposition}\label{c21}
Let $X$ be a smooth minimal projective surface which satisfies {\rm (*)}. Let $a:X\to A$ be the Albanese morphism. Assume that the degree of $a$ is coprime with ${\rm char}(k)$, ${\rm dim}(A)=3$ and ${\rm dim}(V^1(\omega_X))=0$ then the map $b:X\to a(X)$ induced by $a$ is birational.
\end{proposition}

\begin{proof}
Let $c:\widetilde{a(X)}\to a(X)$ be the normalization of $a(X)$ and $\mu:X\to \widetilde{a(X)}$ the induced morphism. By \cite[Th\'{e}or\`{e}me 5.2.2]{GR71} there is a blow-up $p:a(X)'\to \widetilde{a(X)}$ such that the induced morphism $\mu' :X':=X\times_{\widetilde{a(X)}}a(X)'\to a(X)'$ is finite. Let $\nu: \widetilde{X'}\to X'$ be the normalization of $X'$ and $u:=\mu'\circ\nu$. Since $X$ is smooth we know that $\widetilde{X'}$ is klt, in particular it has rational singularities (cf. \cite{Elkik81}). The situation is as follows.
\begin{center}
\begin{tikzpicture}[scale=1.6]
\node (A) at (0,0) {$a(X)'$};
\node (B) at (2,0) {$\widetilde{a(X)}$};
\node (C) at (4,0) {$a(X)$};
\node (D) at (-2,1) {$\widetilde{X'}$}; 
\node (E) at (0,1) {$X'$}; 
\node (F) at (2,1) {$X$}; 
\path[->,font=\scriptsize]
(A) edge node[above]{$p$} (B)
(B) edge node[above]{$c$} (C)
(D) edge node[above]{$\nu$} (E)
(E) edge node[above]{$q$} (F)
(D) edge node[above]{$u$} (A)
(E) edge node[left]{$\mu'$} (A)
(F) edge node[left]{$\mu$} (B)
(F) edge node[above]{$b$} (C);
\end{tikzpicture} 
\end{center}

It is easy to see that ${\rm deg}(b)={\rm deg}(\mu)={\rm deg}(\mu')={\rm deg}(u)$, so $({\rm deg}(u),{\rm char}(k))=1$. By \cite[Proposition 5.7(2)]{KM98} we have that $\mathcal{O}_{a(X)'}$ is a direct summand of $u_*\mathcal{O}_{\widetilde{X'}}$, then by Grothendieck duality $\omega_{a(X)'}^{\bullet}$ is a direct summand of $Ru_*\omega_{\widetilde{X'}}$. By relative Kawamata-Viehweg vanishing (cf. \cite[2.2.5]{KK}), $R^iu_*\omega_{\widetilde{X'}}=0$ for $i\ne 0$, so $Ru_*\omega_{\widetilde{X'}}=u_*\omega_{\widetilde{X'}}$ is a sheaf and $\omega_{a(X)'}$ is a direct summand of it.

Assume that $b$ is not birational. Then the other direct summand is not trivial and we denote it by $\mathcal{F}$, i.e. $u_*\omega_{\widetilde{X'}}=\omega_{a(X)'}\oplus\mathcal{F}$. Then 
$$(b\circ q\circ \nu)_*\omega_{\widetilde{X'}}=(c\circ p\circ u)_*\omega_{\widetilde{X'}}=(c\circ p)_*\omega_{a(X)'}\oplus (c\circ p)_*\mathcal{F},$$ 
and 
$$1=\chi(\omega_X)=\chi(\omega_{\widetilde{X'}})=\chi((b\circ q\circ \nu)_*\omega_{\widetilde{X'}})=\chi((c\circ p)_*\omega_{a(X)'})+\chi((c\circ p)_*\mathcal{F})$$ 
as $b\circ q\circ \nu$ is generically finite.

Now we prove that neither $\chi((c\circ p)_*\mathcal{F})$ nor $\chi((c\circ p)_*\omega_{a(X)'})$ is $0$. Otherwise, without loss of generality we assume $\chi((c\circ p)_*\mathcal{F})=0$. Since 
$$0={\rm dim}(V^i(\omega_X))={\rm dim}(V^i(\omega_{\widetilde{X'}}))={\rm dim}(V^i((b\circ q\circ \nu)_*\omega_{\widetilde{X'}}))\ge {\rm dim}(V^i((c\circ p)_*\mathcal{F}))$$ 
for $i=1,2$ we know ${\rm dim}(V^1((c\circ p)_*\mathcal{F}))={\rm dim}(V^2((c\circ p)_*\mathcal{F}))=0$. This together with $\chi((c\circ p)_*\mathcal{F})=0$ implies that ${\rm dim}(V^0((c\circ p)_*\mathcal{F}))=0$.
Then following the first paragraph of the proof of Theorem \ref{f2} we can get that ${\rm dim}(A)=2$ which is a contradiction to the assumption ${\rm dim}(A)=3$.
 
So $\chi((c\circ p)_*\mathcal{F})$ and $\chi((c\circ p)_*\omega_{a(X)'})$ are both $\ge 1$. But this implies that $\chi(\omega_{X})\ge 2$ which is impossible. So $\mathcal{F}=0$ and $u$ is birational, hence $b$ is birational.
\end{proof}
\bibliographystyle{alpha}
\bibliography{P}  
\end{document}